\providecommand{\href}[2]{#2}

\documentclass[a4paper, 12pt]{amsart}

\usepackage[inner=30mm, outer=30mm, textheight=245mm]{geometry}

\usepackage{graphicx, subfigure, color}
\usepackage{amsmath, amsfonts, amssymb, latexsym, euscript}
\usepackage{mathptmx, wrapfig, url}

\theoremstyle{plain}
\newtheorem{Thm}{Theorem}

\newtheorem{Coro}[Thm]{Corollary}
\newtheorem{Lem}[Thm]{Lemma}

\newtheorem{Prop}[Thm]{Proposition}

\theoremstyle{definition}
\newtheorem{Rem}[Thm]{Remark}
\newtheorem{Exa}[Thm]{Example}
\newtheorem{Def}[Thm]{Definition}

\def\H{\mathbb{H}}
\def\R{\mathbb{R}}
\def\Z{\mathbb{Z}}
\def\C{\mathbb{C}}
\def\Q{\mathbb{Q}}
\def\bd{\partial}

\DeclareMathOperator{\Int}{Int}

\DeclareMathOperator{\im}{im}

\newcommand{\co}{\colon\thinspace}

\def\tri{\mathcal{T}}

\newif\ifcolour
\colourtrue

\begin{document}

\title[Triangulations of 3--manifolds with essential edges]{Triangulations of 3--Manifolds with essential edges}

\author[Hodgson, Rubinstein, Segerman and Tillmann]{Craig D.\thinspace Hodgson, J.\thinspace Hyam Rubinstein, Henry Segerman and Stephan Tillmann}

\address{CH \& HR: Department of Mathematics and Statistics\\
         The University of Melbourne\\
         VIC 3010, Australia}

\address{HS: Department of Mathematics \\
         Oklahoma State University \\
         Stillwater,  Oklahoma, OK 74078, USA}
         
\address{ST: School of Mathematics and Statistics \\
         The University of Sydney \\
         NSW 2006,  Australia}
                  
\email{craigdh@unimelb.edu.au}
\email{rubin@ms.unimelb.edu.au}
\email{segerman@math.okstate.edu}
\email{tillmann@maths.usyd.edu.au}

\thanks{This research is supported by the Australian Research Council under the Discovery Projects funding scheme (project DP130103694).}

\begin{abstract}
We define essential and strongly essential triangulations of $3$--manifolds, and give four constructions using different tools (Heegaard splittings, hierarchies of Haken 3--manifolds, Epstein-Penner decompositions, and cut loci of Riemannian manifolds) to obtain triangulations with these properties under various hypotheses on the topology or geometry of the manifold. 

We also show that a semi-angle structure is a sufficient condition for a triangulation of a 3--manifold to be essential, and a strict angle structure is a sufficient condition for a triangulation to be strongly essential. Moreover, algorithms to test whether a triangulation of a 3--manifold is essential or strongly essential are given.
\end{abstract}


\dedicatory{Dedicated to Michel Boileau for his many contributions and leadership in low dimensional topology.}

\maketitle


\section{Introduction}

Finding combinatorial and topological properties of geometric triangulations of hyperbolic 3--manifolds gives information which is useful for the reverse process---for example, solving Thurston's gluing equations to construct an explicit hyperbolic metric from an ideal triangulation. 
Moreover, suitable properties of a triangulation can be used to deduce facts about the variety of representations of the fundamental group of the underlying $3$-manifold into $\text{PSL}(2,\C)$.  See for example  \cite{segerman_tillmann_11}. 

In this paper, we focus on one-vertex triangulations of closed manifolds and ideal triangulations of the interiors of compact manifolds with boundary. Two important properties of these triangulations are introduced. For one-vertex triangulations in the closed case, the first property is that no edge loop is null-homotopic and the second is that no two edge loops are homotopic, keeping the vertex fixed. We call a triangulation with the first property \emph{essential} and with both properties \emph{strongly essential}. Similarly, for ideal triangulations of the interiors of compact manifolds, we say that the triangulation is \emph{essential} if no ideal edge is homotopic into a vertex neighbourhood keeping its ends in the vertex neighbourhood, and is \emph{strongly essential} if no two ideal edges are homotopic, keeping their ends in the respective vertex neighbourhoods. Precise definitions are given in \S\ref{sec:Essential and strongly essential}.\\

Four different constructions are given of essential and strongly essential triangulations, each using a different condition on the underlying manifold. Each construction is illustrated with key examples.

The first construction applies to closed $P^2$-irreducible 3--manifolds different from $\R P^3$ and with non-zero first homology with $\Z_2$ coefficients. For these manifolds, an essential triangulation is constructed using one-sided Heegaard splittings. A similar construction, using Heegaard splittings with compression bodies each containing a torus boundary component, gives essential ideal triangulations of the interiors of compact orientable irreducible atoroidal 3--manifolds with boundary consisting of two incompressible tori. 

The second construction gives strongly essential triangulations of closed Haken 3--manifolds using hierarchies. A variation of this construction produces strongly essential ideal triangulations of the interiors of compact irreducible atoroidal 3--manifolds with boundary consisting of incompressible tori.

The third construction shows how strongly essential triangulations of non-compact complete hyperbolic $3$--manifolds can be obtained by subdividing the polyhedral decompositions of Epstein and Penner~\cite{EP}.  

In the fourth construction, a one-vertex strongly essential triangulation of a closed $3$--manifold with a Riemannian metric of constant negative or zero curvature is constructed from the dual of the cut locus of an arbitrary point. 
The same method applies to some classes of elliptic closed $3$--manifolds, i.e those with constant positive curvature. However in the latter case, there are restrictions on the fundamental groups. 

We expect that these constructions will extend to higher dimensional manifolds, but in this paper we focus on dimension three.\\

We also will show that an ideal triangulation of a 3--manifold admitting a strict angle structure must be strongly essential and if a triangulation has a semi-angle structure, then it must be essential. The class of semi-angle structures contains the class of taut structures. Moreover,  examples of taut triangulations which are not strongly essential are given, as well as an example of a taut triangulation which is strongly essential, but does not support a strict angle structure. 

In the last section, we show that there are algorithms to decide whether a given triangulation of a closed 3--manifold is essential or strongly essential, based on the fact that the word problem is uniformly solvable in the class of fundamental groups of compact, connected 3--manifolds. Similarly, building on work of Friedl and Wilton~\cite{FW}, we show that there is an algorithm to test whether an ideal triangulation is essential based on the fact that the subgroup membership problem is solvable in the same generality. 
However, to test whether an ideal triangulation is strongly essential, we require the double coset membership problem to be solvable for pairs of peripheral subgroups. Building on work of Aschenbrenner, Friedl and Wilton~\cite{AFW}, we restrict this algorithm to the following three classes: topologically finite hyperbolic 3--manifolds, compact Seifert fibred spaces, and
compact, irreducible, orientable, connected 3--manifolds with incompressible boundary consisting of a disjoint union of tori.\\

We would like to thank the referee for many detailed suggestions to improve the exposition. 


 \section{Preliminaries}
 \label{sec:Preliminaries} 
 
The purpose of this section is to establish the basic definitions and to recall basic properties of $3$--manifolds. Our set-up is somewhat rigid in order to describe objects, such as the compact core, in an algorithmically constructible fashion.
 

\subsection{Triangulations}
\label{sec:Triangulations}
 
Let $\widetilde{\Delta}$ be a finite union of pairwise disjoint Euclidean $3$--simplices with the standard simplicial structure. Every $k$--simplex $\tau$ in $\widetilde{\Delta}$ is contained in a unique $3$--simplex $\sigma_\tau.$ A $2$--simplex in $\widetilde{\Delta}$ is termed a \emph{facet}.

Let $\Phi$ be a family of affine isomorphisms pairing the facets in $\widetilde{\Delta},$ with the properties that $\varphi \in \Phi$ if and only if $\varphi^{-1}\in \Phi,$ and every facet is the domain of a unique element of $\Phi.$ The elements of $\Phi$ are termed \emph{face pairings}.

Consider the quotient space $\widehat{M} = \widetilde{\Delta}/\Phi$ with the quotient topology, and denote the quotient map $p\co \widetilde{\Delta} \to \widehat{M}.$ We will make the additional assumption on $\Phi$ that for every $k$--simplex $\tau$ in $\widetilde{\Delta}$ the restriction of $p$ to the interior of $\tau$ is injective. Then
the set of non-manifold points of $\widehat{M}$ is contained in the $0$--skeleton, and in this case $\widehat{M}$ is called a closed $3$--dimensional pseudo-manifold. (See Seifert-Threfall~\cite{SeiThr}.) We will also always assume that $\widehat{M}$ is connected.
The triple $\tri = ( \widetilde{\Delta}, \Phi, p)$ is a \emph{(singular) triangulation} of $\widehat{M},$ but for brevity we will often simply say that $\widehat{M}$ is given with the structure of a triangulation.

Let $\widehat{M}^{(k)}= \{ p(\tau^k) \mid \tau^k \subseteq \widetilde{\Delta}\}$ denote the set of images of the $k$--simplices 
of $\widetilde{\Delta}$ under the projection map. Then the elements of $\widehat{M}^{(k)}$ are precisely the equivalence classes of $k$--simplices in $\widetilde{\Delta}.$
The elements of $\widehat{M}^{(0)}$ are termed the \emph{vertices} of $\widehat{M}$ and the elements of $\widehat{M}^{(1)}$ are the \emph{edges}. 
The triangulation is a \emph{$k$--vertex triangulation} if $\widehat{M}^{(0)}$ has size $k.$ 


\subsection{Edge paths and edge loops}
\label{sec:Edge paths and edge loops}

To each edge $e$ in $\widehat{M}^{(1)}$ we associate a continuous path $\gamma_e \co [0,1]\to \widehat{M}$ with endpoints in $\widehat{M}^{(0)}$ via the composition of maps $[0,1] \to \widetilde{\Delta} \to \widehat{M},$ where the first map is an affine map onto a 1--simplex in the equivalence class and the second is the quotient map. We call $\gamma_e$ an \emph{edge path}. An edge path is termed an \emph{edge loop} if its ends coincide. For instance, in a 1--vertex triangulation every edge path is an edge loop.

We often abbreviate ``edge path" or ``edge loop" to ``edge," and we denote the edge path $\gamma_e$ with reversed orientation by $-\gamma_e.$ The notions of interest in this paper are independent of the parametrisation except that we often need to take care of the fact that we have chosen an arbitrary orientation of each edge.


\subsection{Ideal triangulations}
\label{sec:Ideal triangulations}

Let $M = \widehat{M}\setminus \widehat{M}^{(0)}.$ Then $M$ is a topologically finite, non-compact 3--manifold, and the pseudo-manifold $\widehat{M}$ is referred to as the \emph{end-compactification} of $M,$ the elements of $\widehat{M}^{(0)}$ as the \emph{ideal vertices} of $M$ and the elements of $\widehat{M}^{(1)}$ as the \emph{ideal edges} of $M.$ 
We also refer to the triangulation $\tri = ( \widetilde{\Delta}, \Phi, p)$ of $\widehat{M}$ as an \emph{ideal (singular) triangulation} of $M.$

The \emph{compact core} $M^c$ of $M$ is the result of taking the second barycentric subdividivision of $\widehat{M}$ and deleting the open star of each vertex. The components of $\widehat{M} \setminus M^c$ are referred to as the \emph{ends} of $M.$ The dual $2$--skeleton $X$ in $\widehat{M}$ is a \emph{spine} for $M^c$ (see \S\ref{sec:spine}). Hence 
$\pi_1(X) = \pi_1(M^c) = \pi_1(M),$ whilst the fundamental group of $\widehat{M}$ is a quotient thereof.


\subsection{Hats off}

The adjective \emph{singular} is usually omitted: we will not need to distinguish between simplicial or singular triangulations. If $\widehat{M}$ is a closed $3$--manifold, we may write $M = \widehat{M},$ and hope this will not cause any confusion. We will suppress the notation $\tri = ( \widetilde{\Delta}, \Phi, p)$ and simply say that $\widehat{M}$ is triangulated or $M = \widehat{M}\setminus \widehat{M}^{(0)}$ is ideally triangulated.


\subsection{Surfaces in 3--manifolds}

Throughout, $3$--manifolds will be compact, connected and we will work in the PL category. A $3$--manifold $N$ is called \emph{irreducible} if every embedded $2$--sphere bounds a $3$--ball; it is $P^2$--irreducible if $N$ is irreducible and does not contain any embedded 2-sided projective planes. An embedded 2-sided surface $S$ in $N$ is called \emph{incompressible} if $S$ is not simply connected and the inclusion map induces an injection $\pi_1(S) \to \pi_1(N)$.
If $S$ is properly embedded in $N$ then $S$ is called \emph{$\partial$-incompressible} if the induced map $\pi_1(S, \partial S) \to \pi_1(N, \partial N)$ is one-to-one. 
A closed $3$--manifold is called \emph{Haken} if it is $P^2$-irreducible and contains an incompressible surface. If $N$ has non-empty boundary, then it is Haken if it is $P^2$--irreducible and each component of $\partial N$ is incompressible. 
Moreover, $N$ is \emph{atoroidal} if any incompressible torus or Klein bottle is isotopic to a component of $\partial N$; $N$ is \emph{anannular} if it has non-empty boundary and contains no incompressible and $\partial$-incompressible annuli or M\"obius bands.  


\subsection{Heegaard splittings}

A \emph{handlebody} is a regular neighbourhood of an embedded graph in a 3--manifold $N$ and a \emph{compression body} is a regular neighbourhood of the connected union of a finite collection of graphs and some boundary components of $N$. 

A \emph{Heegaard surface} in a compact $3$--manifold $N$ is a 2--sided, properly embedded, connected, closed surface $S$ in $N,$ which splits $N$ into two handlebodies or compression bodies $H_1, H_2$. 
The decomposition $N = H_1 \cup_S H_2$ is called a \emph{Heegaard splitting} of $N.$ For emphasis, we will sometimes say that this is a \emph{2--sided} Heegaard splitting.
Note that if $N$ is non-orientable, so are $H_1,H_2$.

A \emph{1--sided Heegaard surface} is a 1--sided, properly embedded, connected, closed surface $S$ in $N$ such that the complement of an open regular neighbourhood $U(S) \subset N$ is a handlebody or compression body $H$. 
The decomposition $N = U(S) \cup_{\partial U(S)} H$ is called a \emph{1--sided Heegaard splitting} of $N.$ In this case, $N$ can be viewed as the result of identifying points of one of the boundary surfaces of a compression body or handlebody under a free involution.
Note that if $N$ is non-orientable, then so are $S$ and $H$. 

A (1-- or 2--sided) Heegaard surface $S$ in $N$ is \emph{irreducible} if there is no separating embedded $2$-sphere in $N,$ which intersects $S$ in a single loop that is essential on $S$. We will also say that the associated Heegaard splitting is irreducible. If $N$ is irreducible and there is such a $2$-sphere $\Sigma$ meeting $S$ in a single essential loop, then $\Sigma$ bounds a $3$-ball $B$ in $N,$ which meets $S$ in a non-simply connected region. By \cite{wald}, the Heegaard splitting of $B$ given by $B \cap S$ is standard, consisting of trivial handles.
Equivalently, $B \cap S$ is isotopic to the boundary of a regular neighbourhood of the union of $\partial B$ and an unknotted collection of properly embedded arcs in $B$.


\subsection{Spines}
\label{sec:spine}

To construct triangulations, we will make use of \emph{spines} of 3--manifolds. We follow the terminology of the monograph by Matveev~\cite{Mat}. 
A spine of the compact manifold $N$ with non-empty boundary is a compact subpolyhedron $P$ of $N$ with the property that $N$ collapses to $P.$ Typically $\dim P= \dim N -1.$ 
We will usually arrange for $P \subset \Int(N),$ so that $P$ is a spine if and only if $N \setminus P$ is homeomorphic with $\partial N \times [0,1)$ (see \cite{Mat}, Theorem~1.1.7). A spine of a closed 3--manifold $N$ is a spine of the complement of an open ball in $N.$ A spine of a topologically finite non-compact 3--manifold $N$ is a spine of a compact core of $N.$

If $M = \widehat{M}\setminus \widehat{M}^{(0)}$ is given with the structure of an ideal triangulation, there is a natural spine dual to the triangulation obtained by inserting a so-called \emph{butterfly} in each tetrahedron. The butterfly is the intersection of the dual 2--skeleton with the tetrahedron (see Figure~\ref{butterfly.pdf}). 

\begin{figure}[htbp]
\centering
\includegraphics[width=0.4\textwidth]{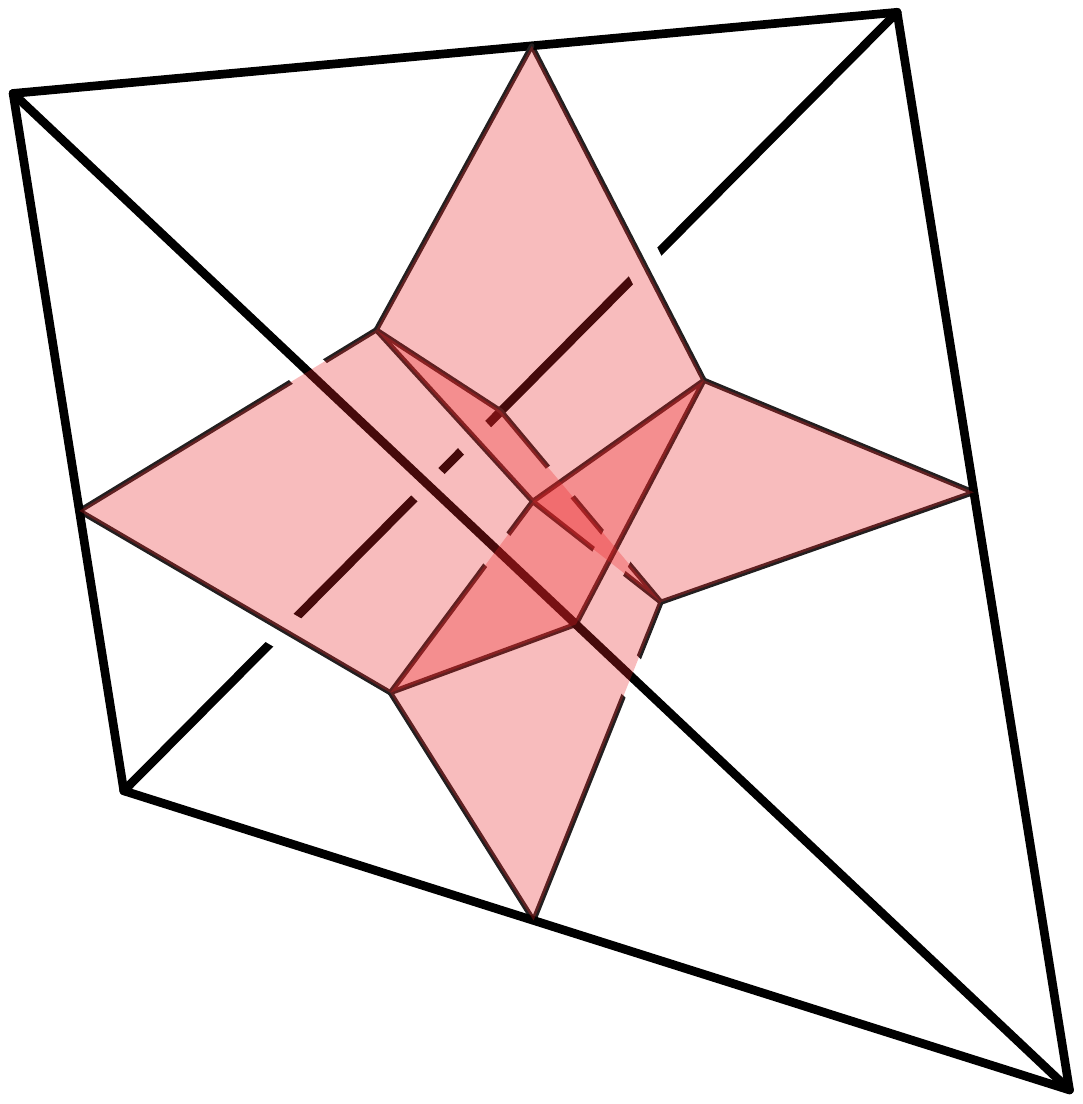}
\caption{A butterfly in a tetrahedron.}
\label{butterfly.pdf}
\end{figure}

A spine is \emph{simple} if every point on the spine has a neighbourhood homeomorphic to a neighbourhood of a point in the interior of the butterfly. The points on a simple spine fall into three categories: \emph{(true) vertices} have arbitrarily small neighbourhoods homeomorphic to the butterfly, 
\emph{non-singular points} have a neighbourhood homeomorphic to a disc, and all other points are \emph{triple points} (see Figure~\ref{point_types_on_spine.pdf}). A connected component of the set of all non-singular points is called a \emph{2--stratum}; a connected component of the set of all triple points is called a \emph{1--stratum}. A simple spine is \emph{special} if each  \emph{2--stratum} is an open 2--cell and each  \emph{1--stratum} is an open 1--cell. 

\begin{figure}[htbp]
\centering
\includegraphics[width=0.8\textwidth]{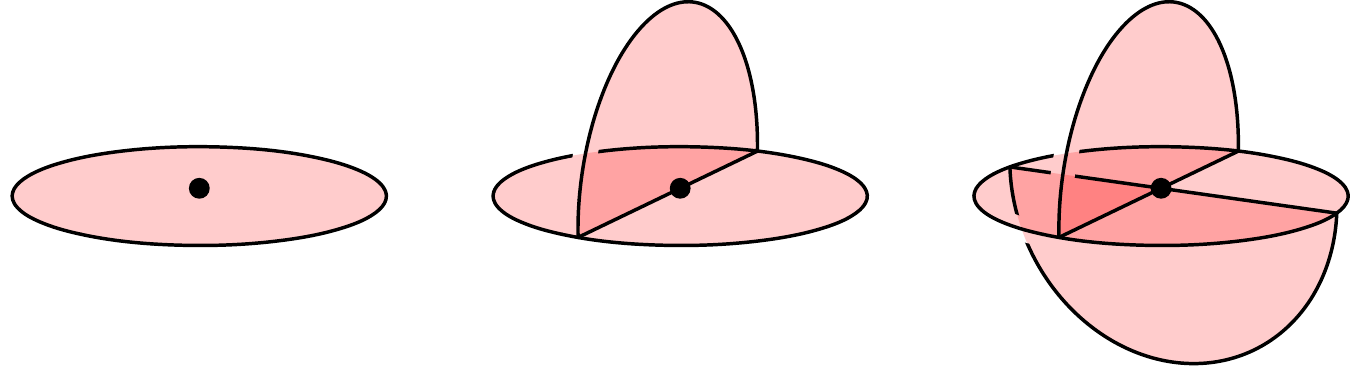}
\caption{From left to right: a non-singular point, a triple point, and a true vertex.}
\label{point_types_on_spine.pdf}
\end{figure}

In our existence proofs for triangulations with certain properties, we will often first construct a special spine and then a triangulation dual to the spine. A general theory for constructing triangulations dual to \emph{thickenable special polyhedra} is given in \cite{Mat}, \S1.1.5. In our setting, we start with a simple spine of a manifold, so know from the outset that we have a thickenable polyhedron, and only need to show that it is special in order to justify the existence of a dual triangulation. As pointed out in \cite{Mat} (Remark 1.1.11), given a simple polyhedron with at least one true vertex, the condition on the 1--strata for a special spine follows from the condition on the 2--strata. 


 \section{Essential and parallel edges}
\label{sec:Essential and strongly essential}
 
The purpose of this section is to establish the definitions of essential and strongly essential triangulations of $3$--manifolds. We are interested in two distinct cases in this paper: 1--vertex triangulations of a closed pseudo-manifold $\widehat{M}$ (which will generally be a closed manifold), and ideal triangulations of a topologically finite manifold $M$ with any number of ideal vertices. 
This section continues in the notation of \S\ref{sec:Preliminaries}, and we will always assume that $\widehat{M}$ is triangulated and $M = \widehat{M}\setminus \widehat{M}^{(0)}$ is ideally triangulated.


\subsection{Closed pseudo-manifolds}
\label{sec:Essential and strongly essential--closed}

The definitions for a closed pseudo-manifold $\widehat{M}$ are straight forward:

\begin{Def}[Essential and parallel edges]
An edge loop in $\widehat{M}$ is \emph{essential}  if it is not null-homotopic. The edge paths $\gamma,$ $\delta$ in $\widehat{M}$ are \emph{parallel} if there is a path homotopy between $\gamma$ and one of $\pm \delta.$\\
\end{Def}

\begin{Def}[Essential triangulation]
The triangulation of $\widehat{M}$ is \emph{essential} if it has exactly one vertex and every edge loop in $\widehat{M}$ is essential. 
It is \emph{strongly essential} if, in addition, no two edge paths are parallel. 
\end{Def}

In the case of an ideal triangulation, we will restrict the allowable path homotopies. For instance, Thurston's ideal triangulation of the figure eight knot complement $M$ yields a simply connected pseudo-manifold $\widehat{M}$ since the fundamental group is normally generated by meridians, and hence the triangulation of $\widehat{M}$ is not essential (and hence not strongly essential). 
However, the \emph{geometry} of the figure eight knot complement begs for a definition that renders this triangulation strongly essential. 

We will first give a definition that displays the similarity between the ideal case and the closed case, and then give equivalent conditions in terms of the compact core, which are easier to verify in practice.


\subsection{Ideal triangulations via pseudo-manifolds}
\label{sec:Essential and strongly essential--ideal}

The key to our definition in the ideal case is that intermediate paths in the homotopy are not allowed to pass through the ideal vertices, except for their endpoints. The reader is reminded than an ideal edge of $M$ is an edge of $\widehat{M}.$

\begin{Def}[Admissible path homotopy]
A path homotopy\footnote{i.e. a homotopy keeping endpoints fixed}
 $H\co [0,1] \times [0,1]\to \widehat{M}$ between two edge paths is \emph{admissible} if $H(\; (0,1)\times [0,1]\;) \subset M,$ where the first factor parametrises the paths. The edge path $\gamma$ is \emph{admissibly null-homotopic} if there is a path homotopy $H\co [0,1] \times [0,1]\to \widehat{M}$ with $H(x, 0) = \gamma(x),$ $H(x, 1) = \gamma(0)$ for all $x\in [0,1]$ and $H(\; (0,1)\times [0,1)\;) \subset M.$
\end{Def}

\begin{Def}[Essential and parallel ideal edges]
An ideal edge of $M$ is \emph{essential} it is not admissibly null-homotopic in $\widehat{M}$. Two ideal edges $e,$ $f$ of $M$ are \emph{parallel} if there is an admissible path homotopy between $\gamma_e$ and one of $\pm \gamma_f.$
\end{Def}

\begin{Def}[Essential ideal triangulation]
The ideal triangulation of $M$ is \emph{essential} if every ideal edge of ${M}$ is essential.
It is \emph{strongly essential} if, in addition, no two ideal edges of $M$ are parallel.
\end{Def}

We remark that whilst the definition of an essential triangulation asks for exactly one vertex, the corresponding ``minimality" condition is built into the definition of an ideal triangulation---having one ideal vertex for each end of $M.$


\subsection{Ideal triangulations via compact core}

We now give equivalent
conditions using the intersection of edges with the compact core $M^c$ of $M$. This alternative viewpoint for essential edges can be found in \cite{segerman_tillmann_11}.

\begin{Lem}
An ideal edge $e$ of $M$ is essential if and only if there is no path homotopy $H \co [0,1] \times [0,1] \to M^c$ between $\gamma_e \cap M^c$  (suitably parametrised) and a path in $\partial M^c.$ 
\end{Lem}

\begin{proof}
We write $\gamma= \gamma_e.$ Suppose there is a path homotopy $H \co [0,1] \times [0,1] \to M^c$ between $\gamma \cap M^c$ and a path $\beta$ in $\partial M^c.$ Coning the endpoints of $\beta$ to the vertex corresponding to the component of  $\partial M^c$ containing $\beta$ gives a loop $\beta'$ in $\widehat{M},$ which is admissibly null-homotopic. The path homotopy between $\gamma \cap M^c$ and $\beta$ extends to an admissible path homotopy between $\gamma$ and $\beta',$ hence showing that $\gamma$ is admissibly null-homotopic.

Conversely, suppose that $\widehat{H} \co [0,1] \times [0,1] \to  \widehat{M}$ is an admissible path homotopy between $\gamma$  and a constant loop at a vertex of $\widehat{M}$.  This restricts to a homotopy $H: E =(0,1) \times [0,1) \to M$ and we can assume that $H$  is transverse to the surface $\bd M^c \subset M$.  Then $H^{-1}(\bd M^c)$ is a compact, properly embedded 1-submanifold of  $E$,  and $H^{-1}(\bd M^c) \cap \bd E$ consists of the two points $(a,0), (b,0)$ with $\gamma(a), \gamma(b) \in \bd M^c$.  Let $\alpha$ be the arc in $H^{-1}(\bd M^c)$ joining these two points, 
and let $D$ be the disc in $E$ bounded by $\alpha$ and the line segment in $\bd E$ between these points.  Then the restriction $H |_D : D \to M$ gives a
path homotopy in $M$ between $\gamma \cap M^c$ and a path in $\bd M^c$.  Composing this with a deformation retraction $r : M \to M^c$, defined using the product structure on the ends of $M$, gives a path homotopy in $M^c$ between
$\gamma \cap M^c$ and a path in $\bd M^c$.  
\end{proof}

\begin{Lem}\label{lem:parallel ideal edges in core}
Two ideal edges $e$ and $f$ of $M$ are parallel if and only if there is a homotopy $H \co [0,1] \times [0,1] \to M^c$ between $\gamma_e \cap M^c$ and $\gamma_f \cap M^c$ (suitably parametrised) with the property that 
$H(0, t),$ $H(1, t) \in \partial M^c$ for all $t\in [0,1],$ i.e.\thinspace the endpoints of all intermediate paths remain on the boundary.
\end{Lem}

\begin{proof}
Suppose $\gamma=\gamma_e$ and $\delta=\gamma_f$ are edge paths in $\widehat{M}$ with the property that they have the same initial vertex $v$ and the same terminal vertex $w.$ If $\gamma$ and $\delta$ are related by an admissible path homotopy of $\widehat{M},$ then using the conical structure of the ends, one can construct a homotopy with image in $M^c$ taking $\gamma \cap M^c$ to $\delta \cap M^c$ whilst keeping the endpoints of all intermediate paths on the boundary. Conversely, if there is such a homotopy of $M^c$ taking $\gamma \cap M^c$ to $\delta \cap M^c,$ then coning the resulting arcs on $\partial M^c$ to the vertices gives an admissible homotopy of $\widehat{M}$ taking $\gamma$ to $\delta.$ 
\end{proof}

As a final remark, we emphasise that ``essential" and ``strongly essential" apply to the edges of $\widehat{M}$ in two different ways---depending on whether they are considered as edges of $\widehat{M}$ or ideal edges of $M.$ We hope that this will not cause any confusion, and we regret that in \cite{segerman_tillmann_11} ``essential in $\widehat{M}$" stands for ``essential in $M.$"

 
\section{Heegaard splittings} 
 
Our first construction takes as its launching point an irreducible one-sided Heegaard splitting:  
 
\begin{Thm}
Assume that $M$ is a closed $P^2$-irreducible 3--manifold $M$ with the property that $H_1(M, \Z_2) \ne 0$ and $M \ne \R P^3$. Then $M$ admits an essential one-vertex triangulation. 

In particular, any closed non-orientable $P^2$-irreducible 3--manifold has such a triangulation. 
\end{Thm}
 
\begin{proof}
We first claim that $M$ admits an irreducible one-sided Heegaard splitting. We know from \cite{rub78} that $M$ admits a one-sided Heegaard splitting $M= N(K) \cup Y$, where $Y$ is a handlebody and $N(K)$ is a regular neighbourhood of a one-sided surface. (In \cite{rub78}, the focus is on orientable 3--manifolds, but the same argument works in the non-orientable case.) To prove the existence of an irreducible one-sided Heegaard splitting, consider a one-sided Heegaard splitting of smallest genus and suppose that it is not irreducible. Then there is a 2--sphere $S$ meeting the Heegaard surface $K$ in a single essential loop on $K$. Moreover $S$ bounds a 3--ball $B$ in $M$ and the splitting $K \cap B$ of $B$ consists of trivial handles. We can then replace the splitting surface $K$ by $K'=(K \setminus K \cap B) \cup D,$ where $D$ is a disk bounded by $K \cap S$ on $S$. Then $K'$ gives a smaller genus one-sided splitting, a contradiction. 

Hence assume that $M= N(K) \cup Y$ is an irreducible one-sided Heegaard splitting of $M.$ We next note that there is no 2--sphere or 2--sided projective plane meeting $K$ in a single essential curve. Indeed, since $M$ is $P^2$-irreducible there is no such 2--sided projective plane in $M$, and since the Heegaard splitting is irreducible, there is no such sphere.
 
We choose a bouquet of circles $\Gamma$ which is a spine of $Y$ and has the property that every circle is essential in $M.$ To achieve this, we first choose an appropriate free generating set for $\pi_1(Y)$. This is non-empty so long as $\pi_1(M)$ is strictly larger than $\Z_2.$ However, if $\pi_1(M) \cong \Z_2,$ then $M$ is homotopy equivalent to $\R P^3,$ and Perelman's proof of the Geometrisation Theorem implies that $M$ is homeomorphic to  $\R P^3,$ contradicting our hypothesis. Consequently it follows that $\pi_1(M)$ is strictly larger than $\Z_2,$ and so $\Gamma$ is non-empty but some circles may be inessential. To achieve that the core of every one-handle is essential, we perform handle slides in the handlebody $Y$, using the fact that the image of $\pi_1(Y)$ is a non trivial subgroup of $\pi_1(M).$ This constructs the graph $\Gamma$.
 
Next, we pick a non-separating family of meridian discs $D_1, \dots D_k$ for $Y$ dual to each of the $k$ circles of $\Gamma$. Each boundary curve $C_i = \partial D_i$ is a component of the boundary of an annulus $A_i$ or a M\"obius band $B_i$ in $N(K)$. The annulus $A_i$ satisfies $\partial A_i = C_i \cup C_i^\prime$ where $C_i^\prime =A_i \cap K$. The M\"obius band $B_i$ meets $K$ in an essential (possibly singular) orientation-reversing curve and $B_i \cup D_i$ is an immersed 2-sided projective plane. The projective plane theorem in \cite{Ep} allows us to replace this immersed 2-sided projective plane by an embedded one. But as we are assuming that $M$ is $P^2$-irreducible, there is no such projective plane in $M.$ Note that $A_i$ is embedded except that $C_i^\prime$ can have isolated double points. We now form the 2-complex $X = K \cup D_1 \cup A_1 \cup \dots D_k \cup A_k$ and claim this is a special spine of $M$. See Figure \ref{one_sided_heegard_hierarchy.pdf}.

\begin{figure}[htbp]
\centering
\includegraphics[width=0.45\textwidth]{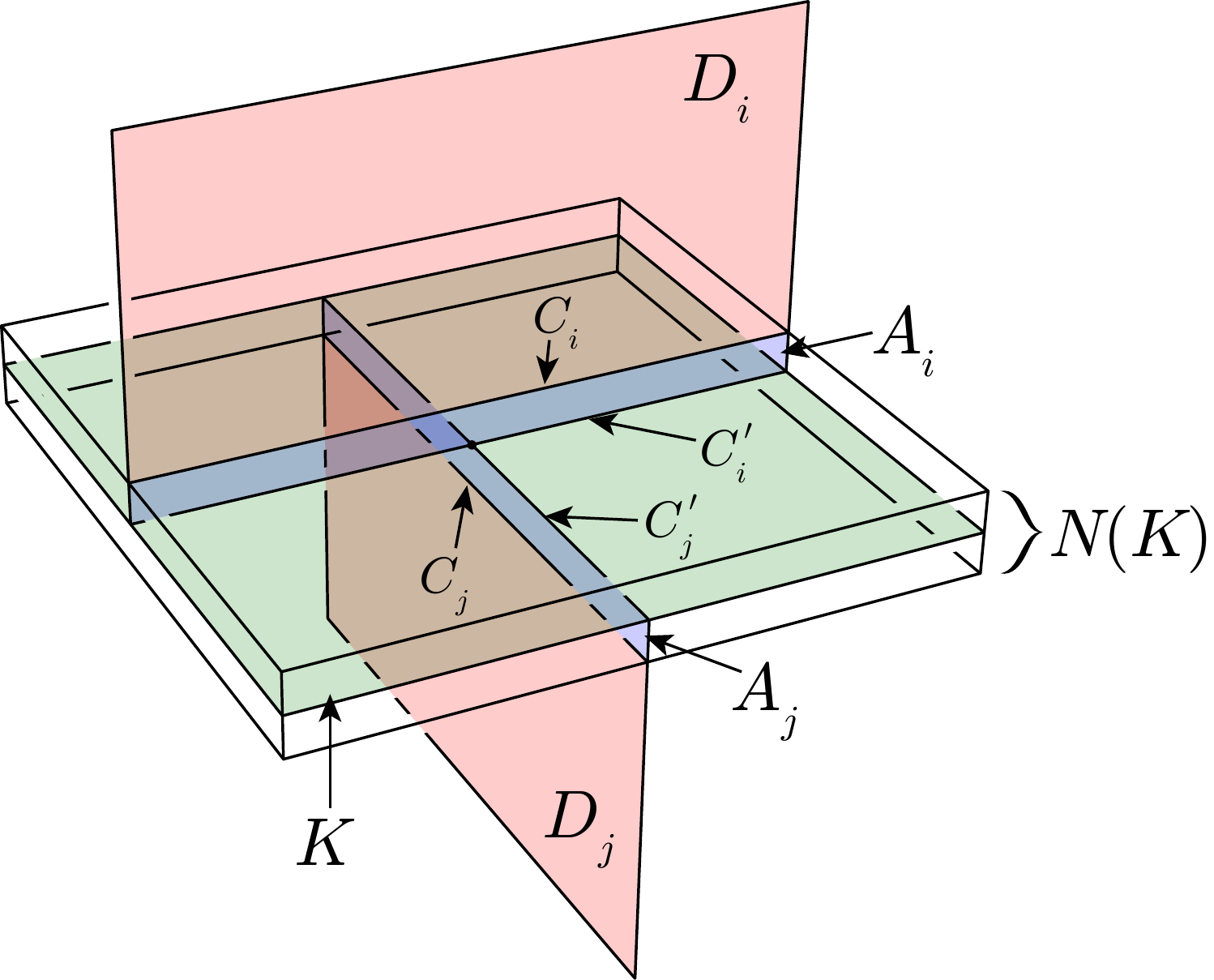}
\caption{Two discs $D_i$ and $D_j$ with their associated annuli $A_i$ and $A_j$ near the one-sided Heegaard surface $K$. Here, $i$ may be equal to $j$.}
\label{one_sided_heegard_hierarchy.pdf}
\end{figure}

It is clear that $M \setminus X$ is a 3--cell, and hence $X$ is a spine of $M.$ Moreover, we may assume that the curves $C_i^\prime$ on $K$ meet transversely, and hence the spine $X$ is simple. If $X$ has no true vertex, i.e.\thinspace no two curves $C_i^\prime$ and $C_j^\prime$ on $K$ meet, we may perform a \emph{lune move} (see Figure~\ref{lune_move.pdf} and \cite{Mat}, \S1.2) to arrange for the existence of true vertices. Denote the resulting simple spine 
with at least one true vertex again by $X.$ We wish to show that $X$ is special. As discussed in \S\ref{sec:spine}, by \cite{Mat} (Remark 1.1.11) it now remains to show that the 2--strata of $X$ are open 2--cells.

\begin{figure}[htbp]
\centering
\includegraphics[width=0.7\textwidth]{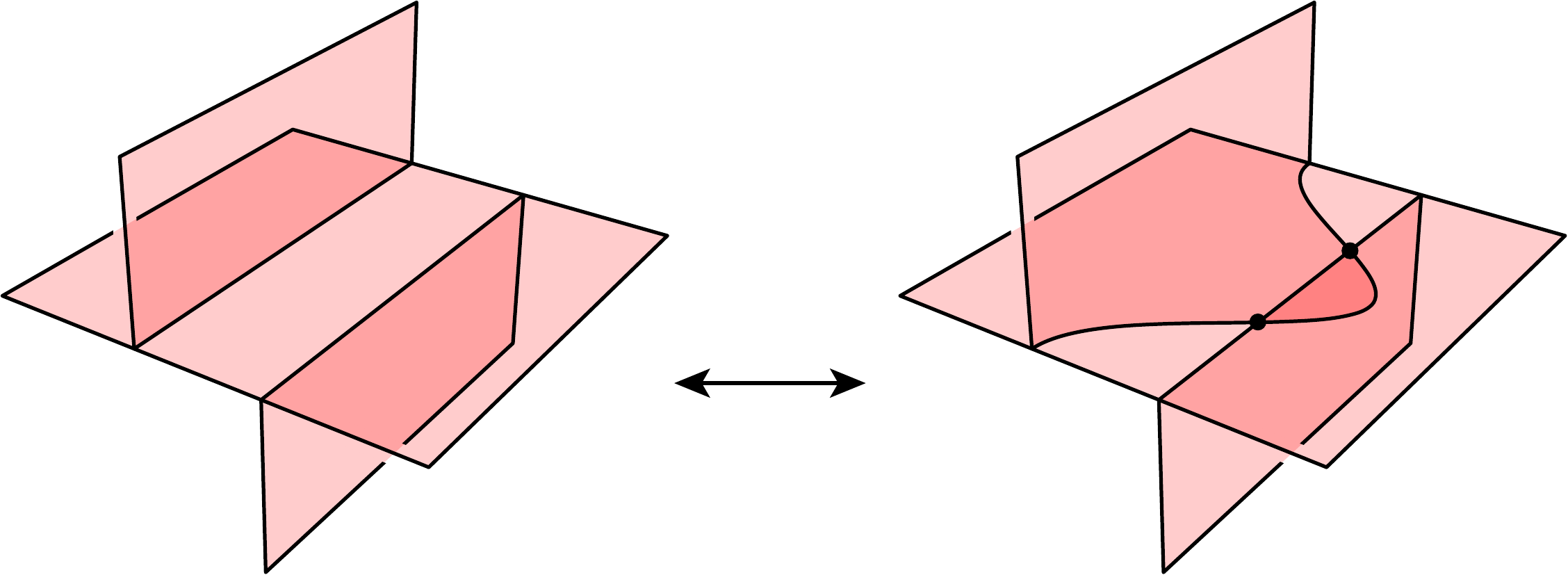}
\caption{The lune move introduces two true vertices.}
\label{lune_move.pdf}
\end{figure}

Since the interior of each $A_i\cup D_i$ is a disc, we need to show that each 2--stratum on $K$ is also a 2--cell. For the sake of contradiction, suppose that we can find a 2--stratum $F$ on $K$, which is not a disc. Choose an essential simple closed curve $C$ in $F$. Assume first that $C$ is 2-sided in $N(K),$ i.e.\thinspace that the lift to $\partial Y$ under the natural map $\partial Y \to K$ has two components. Denote these components by $C_+$ and $C_-$. An annulus $A$ properly embedded in $N(K)$ can be found with $\partial A = C_+ \cup C_-$ and $A \cap K =C$. Since $C \subset F$, it follows that $C_+, C_-$ are both disjoint from $D_1, \dots D_k$. Hence $C_+, C_-$ bound disjoint discs $D_+, D_-$ properly embedded in $Y$. We obtain an embedded 2-sphere $S=D_+ \cup A \cup D_-$ intersecting $K$ in a single essential curve, contrary to assumption. If $C$ is one-sided in $N(K)$ then there is a M\"obius band $B$ properly embedded in $N(K)$ with $B \cap K = C$. Then $\partial B$ bounds a properly embedded disc $D$ in $Y$ and $B \cup D$ is a 2-sided projective plane, contrary to assumption. 

As in \cite{Mat} \S1.1.5, there is a triangulation $\mathcal T$ of $M$ dual to $X$ with precisely one vertex. We claim that this triangulation is essential. The edge loops are either dual to the discs $D_1, \dots D_k$ (and hence homotopic to the circles of $\Gamma$) or they are dual to faces $F$ of $K$ in $X$ (and hence homologically non-trivial). In either case they are homotopically non-trivial and therefore essential. This completes the proof of the main statement. 

The last sentence follows since, as is well-known, $H_1(M, \Z_2) \ne 0$ for any closed non-orientable 3--manifold. 
\end{proof}
   
\begin{Exa}[Quaternionic space]
The minimal two tetrahedron triangulation of quaternionic space $S^3/Q_8$ (see Figure~2(d) in \cite{JRT}) illustrates the construction in the above proof. Here, the incompressible one-sided Heegaard surface is a Klein bottle $K,$ and sits as the union of two normal quadrilaterals in the triangulation. There is precisely one edge in the complement of $K,$ corresponding to the single compression disc $D,$ and there are two more edges in the triangulation arising from the components of $K \setminus D.$
\end{Exa}   
   
We next give a variation on the above construction, which applies to the interior of a compact orientable irreducible atoroidal 3--manifold $M$, which has boundary consisting of two incompressible tori. 
 
 \begin{Thm}
Suppose $M$ is a compact orientable irreducible anannular 3--manifold with boundary having two tori components which are incompressible. Also assume that $M$ is not homeomorphic to $T^2 \times [0,1]$. 
Then there is an essential ideal triangulation $\mathcal T$, which is dual to a spine $X$ of $M$ coming from a compression body decomposition. 
 \end{Thm}
 
 \begin{proof}
Choose an irreducible Heegaard decomposition of $M$ so that $M = Y \cup Y^\prime,$ where both $Y$ and $Y^\prime$ are compression bodies with a common boundary component $S$ and one other boundary component being an incompressible torus in $\partial M$.  Since the decomposition is irreducible, there cannot be a 2-sphere intersecting $S$ in a single essential curve. A spine $X$ for $M$ can be built by $X = S \cup D_1 \cup \dots D_k  \cup D^\prime_1 \cup \dots D^\prime_m$ where $D_1, \dots ,D_k$ is a collection of compressing discs for $Y$ and $D^\prime_1, \dots ,D^\prime_m$ is a similar collection for $Y^\prime$. Moreover, each disc collection is chosen such that it cuts its compression body up into a collar of the associated boundary torus, and, as in the previous proof, the boundaries of the discs can be assumed transverse on $S$ and having at least one intersection point. Hence $X$ is a simple spine of $M$ with at least one true vertex. 

Suppose there is a 2--stratum $F$ of $X,$ which is not a disc. Then $F \subset S$ and there must be either a 2--sphere meeting $S$ in a single essential curve or an essential annulus between the two boundary components of $M$ intersecting $S$ in a single loop. Both of these are ruled out since the splitting is irreducible and $M$ is anannular. So $X$ is a special spine of $M$. 

Let $N$ be the open manifold obtained from $M$ by adding a collar $\partial M \times [0,1)$ to its boundary. Then $N$ is homeomorphic to the interior of $M,$ $X$ is a special spine of $N$ with at least one true vertex, and we denote $\mathcal T$ the ideal triangulation of $N$ dual to $X.$ An ideal edge $e$ meets $M$ is a properly embedded arc $\lambda_e,$ and $e$ is essential if and only if $\lambda_e$ is not path homotopic into $\partial M.$

Now $e$ is either dual to a face of $S$ or to a disc in one of the families $\{D_1, \dots ,D_k\}$, $\{D^\prime_1, \dots ,D^\prime_m\}.$
If $e$ is dual to a 2--stratum contained on $S,$ then $\lambda_e$ has endpoints on different boundary tori and hence $e$ is essential. If $e$ is dual to one of the discs, then it may not be essential and we will show that (analogous to the proof of the previous theorem) one may perform a handle slide to arrive at an essential ideal triangulation. Taking advantage of the symmetry of the situation, it suffices to show this for the edges dual to $D_1, \dots ,D_k.$

We first show that there is at least one disc $D_i$ for $Y$ with the property that the arc $\lambda_i$ dual to $D_i$ is not path homotopic into a boundary torus. By way of contradition, suppose this is not the case. Then the image of $\pi_1(Y)$ in $\pi_1(M)$ would have to be the same as the image of $\pi_1(T)$ in $\pi_1(M).$ But then the inclusion $T \subset M$ is a homotopy equivalence and so $M=T \times [0,1]$, contrary to assumption (see \cite{stall}).  Hence $Y$ has at least one non-trivial dual arc to a disc.

Next suppose some disc  $D_j$ has a dual arc $\lambda_j$ path homotopic into the boundary torus. We can perform a handle slide of the handle dual to $D_j$ over the handle dual to $D_i$. This replaces $D_j , \lambda_j$ 
by a new disc/arc pair which is non-trivial, in the sense that the arc is not path homotopic into the boundary torus. 

So by induction on the number of pairs $D_j , \lambda_j$ for which $\lambda_j$ is path homotopic into the boundary torus, we conclude that  $\mathcal T$ is essential after doing handle slides. 
\end{proof}

\begin{figure}[t]
\centering
\includegraphics[width=0.6\textwidth]{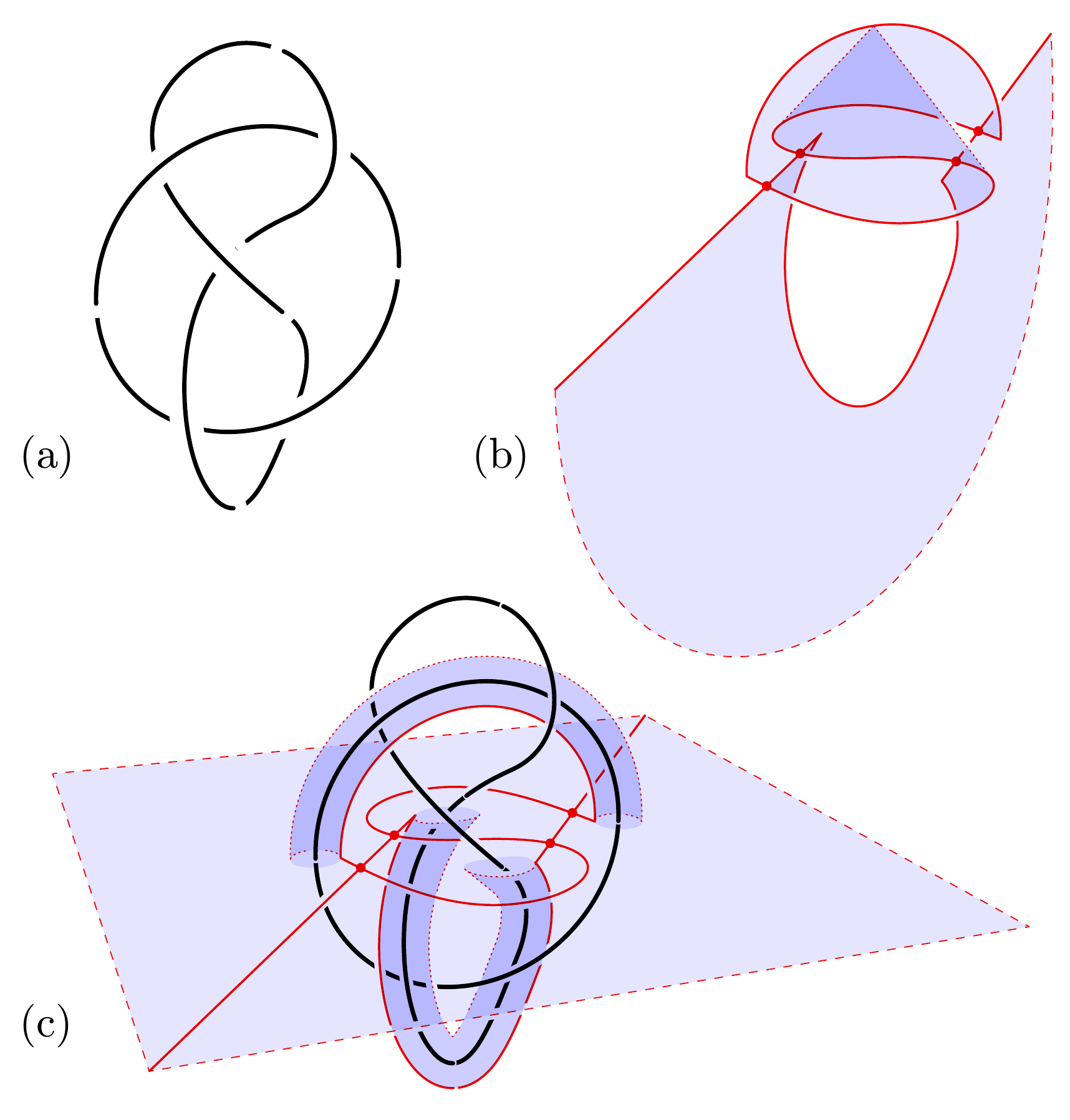}
\caption{}
\label{whitehead_link_example.pdf}
\end{figure}

\begin{Exa}[The Whitehead link]
Refer to Figure \ref{whitehead_link_example.pdf}. Diagram (a) shows the Whitehead link. Diagrams (b) and (c) show a special spine for the complement of the Whitehead link. In (c) we have a horizontal plane, with two tubes, one following one of the link components above the plane, and the other following the other link component below the plane. In (b) we have two discs, one above and one below the plane. Triple points for the spine are shown with a solid line, while the apparent contour of our view of the surface is shown dotted, and the cut out borders of the surface around infinity are shown dashed. The true vertices of the spine are shown with dots. Note that there are four true vertices, so the corresponding triangulation has four ideal tetrahedra. It can be checked that this is the same triangulation as listed in SnapPy~\cite{SnapPy} for the link \verb+S^2_1+.
\end{Exa}


 \section{Hierarchies}
 \setcounter{secnumdepth}{5} 


\begin{Thm}\label{closed haken}
Assume that $M$ is a closed Haken 3--manifold $M$. Then $M$ admits a strongly essential one-vertex triangulation.   
 
In particular, every closed $P^2$-irreducible non-orientable 3--manifold has a strongly essential triangulation. 
\end{Thm}
 
\begin{proof}
We build a hierarchy for $M$ whose complement consists of either one or two open 3--cells. 
This  is closely related to the construction in \cite{raman} of $0$--efficient triangulations of Haken 3--manifolds. There are two cases, depending on whether $H_1(M, \Q) \ne 0$ or $H_1(M, \Q)=0$. The dual of the hierarchy will be a one-vertex or two-vertex triangulation respectively. In the first case, we show that the triangulation is strongly essential and in the second case, we crush an edge which joins the two vertices to a point and then establish that the resulting triangulation is strongly essential. 

\subsection{Case 1: $H_1(M, \Q) \ne 0$}\label{closed, no crushing}

Begin with a closed incompressible non-separating 2-sided surface $S_1$, which exists since $H_1(M, \Q) \ne 0$, using the well-known method of \cite{stall}. Cut open along $S_1$ to give a compact connected manifold $M_1$ with incompressible boundary consisting of two copies of $S_1$. Construct a non-separating 2-sided incompressible and $\partial$-incompressible proper surface $S_2$ in $M_1$ (again using \cite{stall}) and cut open along $S_2$. Note that $M_2$ may not have incompressible boundary.  Iterating this construction, we can find a hierarchy ${\mathcal H} = \{S_1, \dots ,S_k\}$ with the property that $M$ cut open along $\mathcal H$ is a single ball. Note that surfaces $S_i$ in the hierarchy may be compressing disks.

We require that:
\begin{enumerate}
\item \label{item1} $S_i$ is a non-separating 2-sided incompressible and $\partial$-incompressible proper surface (including the possibility of compressing disks) in $M_{i-1}.$
\item \label{item2} $S_i$ is in general position with respect to $\{S_1, \dots ,S_{i-1}\}$. 
\item  \label{item3} For each $i=2,\ldots,k$, $S_i$ is chosen amongst all surfaces in $M_{i-1}$ that satisfy (\ref{item1}) and (\ref{item2}) so that $\partial S_i$ meets the 1-skeleton of the previous surfaces and itself minimally relative to a lexicographical complexity function. That is, we look at the vector 
$$\left(\partial S_i \star \partial S_2, \partial S_i \star \partial S_3, \ldots, \partial S_i \star \partial S_i\right)$$ 
where $\partial S_i \star \partial S_j$ counts the number of transverse intersection points between $\partial S_i$ and $\partial S_j$.
We choose $S_i$ so that this is smallest in the lexicographical ordering.
\end{enumerate}

Note that Waldhausen \cite{wald1} uses a different but related criterion, choosing a surface $S_i$ of maximal Euler characteristic. Conditions (\ref{item2}) and (\ref{item3}) imply that the hierarchy is a special spine and so the dual cellulation to the hierarchy is a triangulation.
The dual triangulation $\mathcal T$  to $\mathcal H$  is then a one-vertex triangulation.

\subsubsection{} We claim that a triangulation dual to a hierarchy satisfying the above three conditions has the property that every edge loop is essential and no two loops are homotopic, keeping the vertex base point fixed.  
 
\subsubsection{} \label{homotope using incomp} Suppose that $\alpha$ is an edge loop. Then by definition of the dual triangulation, $\alpha$ is dual to a surface $S_i$ of the hierarchy, i.e meets this surface in a single point and is disjoint from the other surfaces of the hierarchy. If we cut $M$ open along surfaces $S_1, \dots ,S_{i-1}$, the result is a compact 3--manifold $M_{i-1}$ where $S_i$ is a non-separating 2-sided incompressible and $\partial$-incompressible proper surface. If $\alpha$ is contractible, then there is a map $f:D \to M$, where $D$ is a disc, satisfying $f(\partial D)=\alpha$. Since all the surfaces $S_1, \dots ,S_{i-1}$ are incompressible, we can homotope $f$ so that $f(D)$ misses all these surfaces, starting inductively with a surface of smallest index. But then $f:D \to M_{i-1}$ implies that $\alpha$ is contractible in $M_{i-1}$, which contradicts invariance of algebraic intersection number between $S_i$ and $\alpha$. Hence we have established that the triangulation $\mathcal T$ is essential. 
 
\subsubsection{} \label{pull off homotopy two surfaces} We will now establish the second property, i.e that $\mathcal T$ is strongly essential. Suppose that $\alpha_1, \alpha_2$ are different edge loops. These are dual to surfaces $S_i,S_h$ in the hierarchy. There are two cases, $i \ne h$ and $i=h$. The first case is relatively easy. We can assume without loss of generality that $i<h$. If we cut $M$ open along surfaces $S_1, \dots ,S_{i-1}$, the result is a compact 3--manifold $M_{i-1}$ where $S_i$ is a non-separating 2-sided incompressible and $\partial$-incompressible proper surface. By assumption,  $\alpha_1$ is a dual loop to $S_i$, whereas $\alpha_2$ is disjoint from $S_i$. Hence the loops are not homotopic in $M_{i-1}$. Suppose there was a homotopy $H:D \to M$ between these two edges loops, where $D$ is a disc. Hence $H(\partial D) = \alpha_1 \cup \alpha_2$. As in \ref{homotope using incomp}, we can homotope $H$ off the surfaces  $S_1, \dots ,S_{i-1}$ since these surfaces are incompressible and $\partial$-incompressible. But this would then give a homotopy between $\alpha_1,\alpha_2$ in $M_{i-1}$, which as we noted is impossible. 
This completes the proof in the first case. 
 
\subsubsection{}For the second case $i=h$ and we again cut open along $S_1, \dots ,S_{i-1}$ and examine a homotopy $H:D \to M$ between $\alpha_1, \alpha_2$.  As in the previous case, such a homotopy can be pulled off $S_1, \dots ,S_{i-1}$ so is contained in $M_{i-1}$. By definition, the edge loops $\alpha_1$ and $\alpha_2$ intersect $S_i$ exactly once and are disjoint from $S_1, \dots ,S_{i-1}$, so the pullback of $S_i$ under the homotopy is an arc with two endpoints, one on each of the boundary arcs, which are pre-images of $\alpha_1, \alpha_2$, together with some simple closed curves. By altering the homotopy, we can arrange that the pullback of $S_i$ under the homotopy is a single arc, since any loops in the pullback of $S_i$ can be eliminated using the incompressibility of $S_i$. 

\subsubsection{} Consider now the pullback of the remainder of the hierarchy, i.e the surfaces $S_{i+1}, \dots ,S_k$. Our plan is to change the homotopy so that it misses all of these surfaces, and is still disjoint from $S_1, \dots ,S_{i-1}$. Let $j \in [i+1,k]$ be the least integer such that   $S_j$
has non-empty intersection with the image of the homotopy $H$. The pullback of $S_j$ under $H$ will have arcs with ends on the pullback of $S_i$ and loops in the interior of $D$. No arcs of the pullback of $S_j$ can have ends on $\partial D$, since the image of $\partial D$ is the two edges $\alpha_1$ and $\alpha_2$, whose only intersection with the hierarchy is with $S_i$. We are not yet concerned with the intersections of $D$ with surfaces deeper in the hierarchy.

\begin{figure}[htb]
\centering
\includegraphics[width=0.8\textwidth]{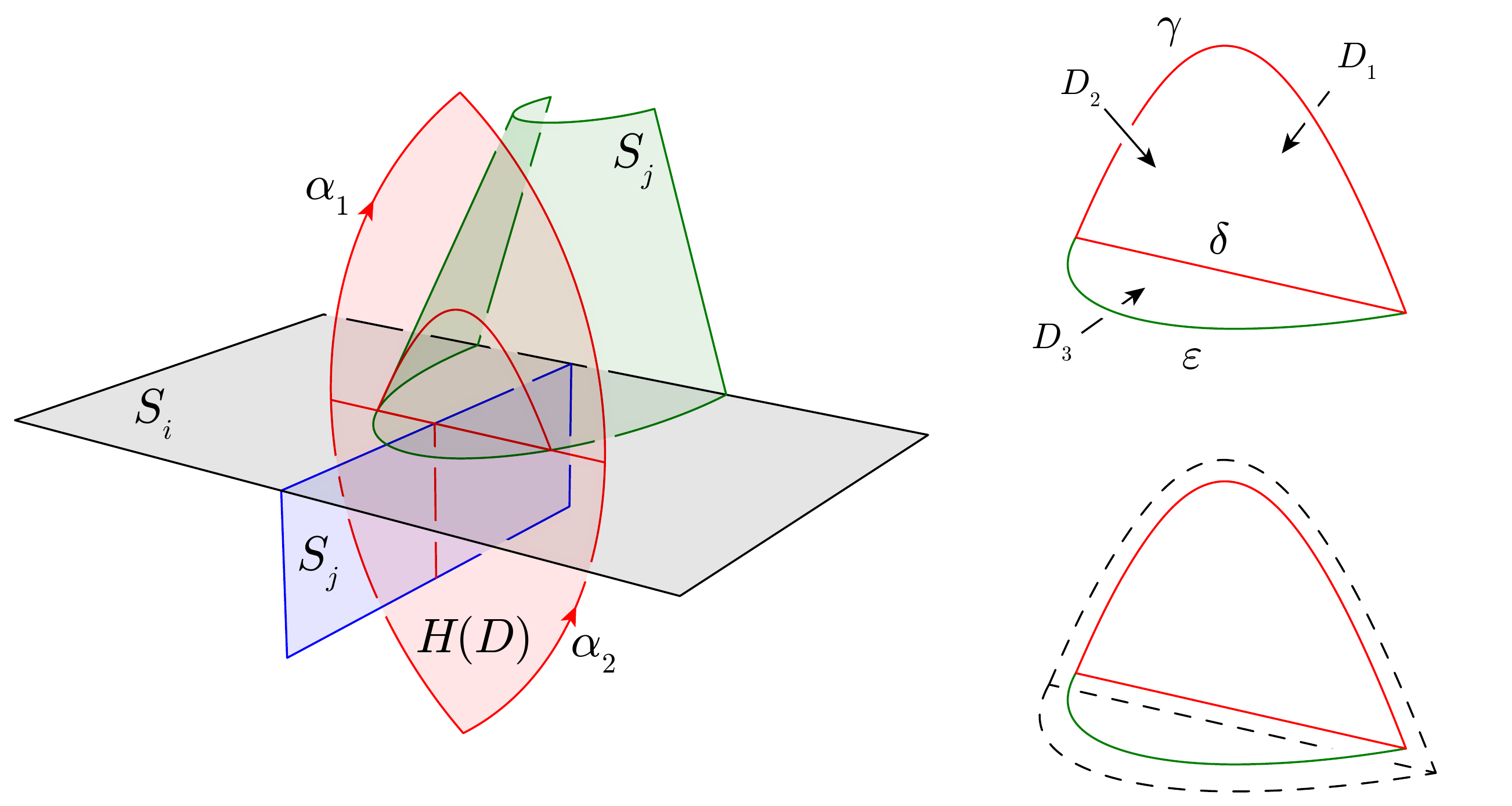}
\caption{An example showing how the image of the homotopy disc can intersect the hierarchy. The diagrams to the right are a close-up view of the intersection arcs on the left, and the way to alter the homotopy to avoid the arc $\gamma$.}
\label{remove_arc.pdf}
\end{figure}

\subsubsection{} \label{remove arcs bigons} Loops can be eliminated in the usual way, using the incompressibility of $S_j$ in $M_{j-1}$. Removing the arcs is more complicated. See Figure \ref{remove_arc.pdf}. Note that the figure shows $H(D)$ embedded in $M$, although in general it will be immersed and the endpoints of $\alpha_1$ and $\alpha_2$ are identified. Also in the argument following, we refer to discs, which in general are images of discs, rather than embeddings. To simplify matters, we will first deal with the special case of $H(D)$ embedded and then indicate how to proceed in the general case. We choose an innermost arc of $H(D) \cap S_j$ on one of the two sides of $S_i$; in the figure this arc is $\gamma$. This bounds a disc $D_1$ on $D$, with $\partial D_1 = \gamma \cup \delta$, and $\delta \subset S_i$. 
Since $S_j$ is $\partial$-incompressible in $M_{j-1}$, the disc $D_1$ in $M_{j-1}$ with boundary the union of $\gamma \subset \partial S_j$ and $\delta \subset \partial M_{j-1}$ (where $\gamma \cap \delta = \partial\gamma = \partial\delta$) implies the existence of a disc $D_2$ in $S_j$ bounded by $\gamma$ and an arc $\epsilon$ in $\partial M_{j-1}$. 

\subsubsection{} \label{epsilon} In fact, the arc $\epsilon$ is contained only in $S_i$. To see this, first note that the only intersections we see with $\delta$ are with arcs of intersection with $S_j$ (i.e. there are no arcs of intersection with $S_{j'}$ for $i<j'<j$). This is because we have already chosen $H(D)$ to miss all of these surfaces. If $\epsilon$ were not contained entirely in $S_i$, then it would cross either the boundary of $S_i$, or the boundary of some other surface $S_{j'}$ with $j'<j$. In either case, we could modify $S_j$ by replacing $D_2$ with $D_1$. The resulting surface would still be incompressible and $\partial$-incompressible, and would have smaller lexicographical complexity since replacing $\epsilon$ with $\delta$ would remove intersections of $\partial S_j$ with the boundary of some surface of the hierarchy with smaller index than $j$.

\subsubsection{} \label{inside outside} Thus $\delta \cup \epsilon$ is a loop in $S_i$. Since $S_i$ is incompressible, $\delta \cup \epsilon$ bounds a disc in $S_i$. We rule out the possibility that $D_3$ is on the ``outside''; that is that $\alpha_1$ and $\alpha_2$ intersect $D_3$. The disc $D_3$ together with $D_1$ and $D_2$ form a sphere, which bounds a ball $B$ in $M$ since $M$ is irreducible. Either this ball is above or below the sphere as seen in Figure \ref{remove_arc.pdf}. If it is above then $S_j$ is contained in $B$, if below then $S_i$ is contained in $B$. In either case, a surface contained in a ball is not incompressible.
Therefore the disc $D_3$ is on the ``inside'', as labelled in Figure \ref{remove_arc.pdf}.

\subsubsection{} \label{modify} Now, we can modify the homotopy, eliminating the arc $\gamma$, by replacing $H$ on a small neighbourhood of the pullback of $D_1$ with a map to the union of $D_2$ and $D_3$, pushed slightly off as in the lower right diagram of Figure \ref{remove_arc.pdf}.
It is possible that parts of $S_j$ connect onto the other side of $S_i$ from the arc $\gamma$ (such a part of $S_j$ is shown in the figure), 
and we should be concerned that in removing the (pullback of the) arc $\gamma$, we do not introduce more arcs. There are no arcs on the part of the disc that is pushed off of $D_2$, since such an arc would correspond to a piece of surface with boundary on $S_j$. An arc on the part of the disc that is pushed off of $D_3$ has a corresponding arc of the boundary of $S_j$ in $D_3$ itself. 

\subsubsection{} We claim that the only arcs that are possible in $D_3$ are arcs of intersection with $S_j$ that intersect each of $\delta$ and $\epsilon$ once, or $\delta$ twice. (In Figure \ref{remove_arc.pdf} only one arc that hits both $\delta$ and $\epsilon$ is shown.)  First, there are no arcs of intersection of $D_3$ with surfaces $S_{j'}$ for $i<j'<j$. To see this, note that there are no such intersections with $\delta$ by our choice of $H(D)$, and any such intersections with $\epsilon$ would contradict the minimality of the lexicographical complexity function for $S_j$. If there were any such arcs of intersection entirely in the interior of $D_3$, then the first such surface in the hierarchy, $S_{j'}$, would intersect $S_i$ in a loop bounding a disc, which would contradict incompressibility of $S_{j'}$ (since $S_{j'}$ is not a disc).
Second, any arcs of intersection with $S_j$ that intersect $\epsilon$ twice could be isotoped over to reduce the lexicographical complexity function for $S_j$.
Therefore each arc has at least one of its endpoints on $\delta$, which means that it connects to a pre-existing arc on $H(D)$. So this modification of the homotopy removes the pullback of $\gamma$ from $D$, and does not introduce any new arcs. 
Continuing in this way, we can remove all arcs of intersection between $S_j$ and $H(D)$. 

\subsubsection{}
We can then repeat the process for $S_{j''}$ where $j<j''\leq k$. 
Eventually, we end up with the graph of intersection between $H(D)$ and the hierarchy consisting of just a single arc, which means that the edge loops $\alpha_1, \alpha_2$ are identical in the triangulation. The proof of the first case is now complete assuming that $H(D)$ is embedded.

\subsubsection{} We now indicate the modifications required to deal with the general case when $H(D)$ is singular. Firstly in 4.1.7 we need to explain why the arc $\epsilon$ must be contained in $S_i$. Note that the subdisk $D_1$ of $D$, mapped by $H$, must be homotopic into $S_j$, since $S_j$ is $\partial$-incompressible in $M_{j-1}$. Hence $\delta$ is homotopic to $\epsilon$ in $\partial M_{j-1}$. We claim that this contradicts our choice of $S_j$ satisfying property (3) that $\partial S_j$ has a minimal number of transverse intersection points with boundary curves $\partial S_{j'}$ for $j'<j$ with these numbers lexicographically ordered. 

To prove the claim, notice that the curves $\partial S_{j'}$ for $j'<j$ form a trivalent graph $\Gamma$ (often called a boundary pattern) in $\partial M_{j-1}$. The homotopy between $\delta$ and $\epsilon$ can be easily shown to imply there is an innermost bigon between a subarc of $\epsilon$ and a subarc of $\Gamma$ in case $\epsilon$ is not contained in $S_i$. But then $\epsilon$ can be isotoped across this bigon producing a proper isotopy of $S_j$ in $ M_{j-1}$ reducing the number of intersection points of $\partial S_j$ with the boundary pattern, with the lexicographic ordering of complexity. So this establishes the claim and establishes the step in 4.1.7.

The remainder of the argument in 4.1.8 to 4.1.11 follows as in the embedded case, requiring only the observation that the ball $B$ in 4.1.8 is the image of a ball, rather than an embedded ball. But this does not affect the arguments and so the first case is established. 

\subsection{Case 2: $H_1(M, \Q) = 0$}

In the second case, the first surface $S_1$ in the hierarchy will be a closed separating incompressible 2-sided surface. The remaining surfaces for the two components of $M$ split open along $S_1$ will be properly embedded non-separating incompressible and $\partial$-incompressible 2-sided surfaces. After cutting open along these we get two 3-cells, so the dual triangulation $\mathcal T$ to this hierarchy $\mathcal H$ has two vertices. We claim that the result of crushing any edge $E$ joining the two vertices is a well-defined one-vertex triangulation $\mathcal T^*$ of $M.$ Here the method of \cite{0-efficient} for crushing suitable 3-cells with normal 2-sphere boundary will be followed. To prove this claim, we first establish that the boundary $B$ of a regular neighbourhood of $E$ is a normal 2-sphere. 

\subsubsection{} Suppose $F$ is a face of the triangulation $\mathcal T$ with at least two edges in $\partial F$ identified with $E.$ Then $F$ is either a cone, a M\"obius band, a 3-fold or a dunce hat (see \cite{0-efficient} for the terminology). The last three cases cannot occur since then the ends of $E$ would have to be the same vertex. Hence assume $F$ is a cone. Then the third edge $E^\prime$ of $\partial F$ is a loop which is contractible via $F.$ We can now use the same method as in Case 1 to produce a contradiction: $E^\prime$ is dual to some surface $S_i$ of the hierarchy. The contracting disc $F$ can be pulled off of all the previous surfaces $S_1, \dots S_{i-1}$ of the hierarchy. We then see that  $E^\prime$ is contractible in $M_{i-1}$, and this gives a contradiction to the fact that $E^\prime$ is dual to $S_i$. So this shows that $F$ cannot be a cone either. It follows that each face of the triangulation contains $E$ at most once, and hence the surface $B$ is normal.

\begin{figure}[htb]
\centering
\includegraphics[width=0.7\textwidth]{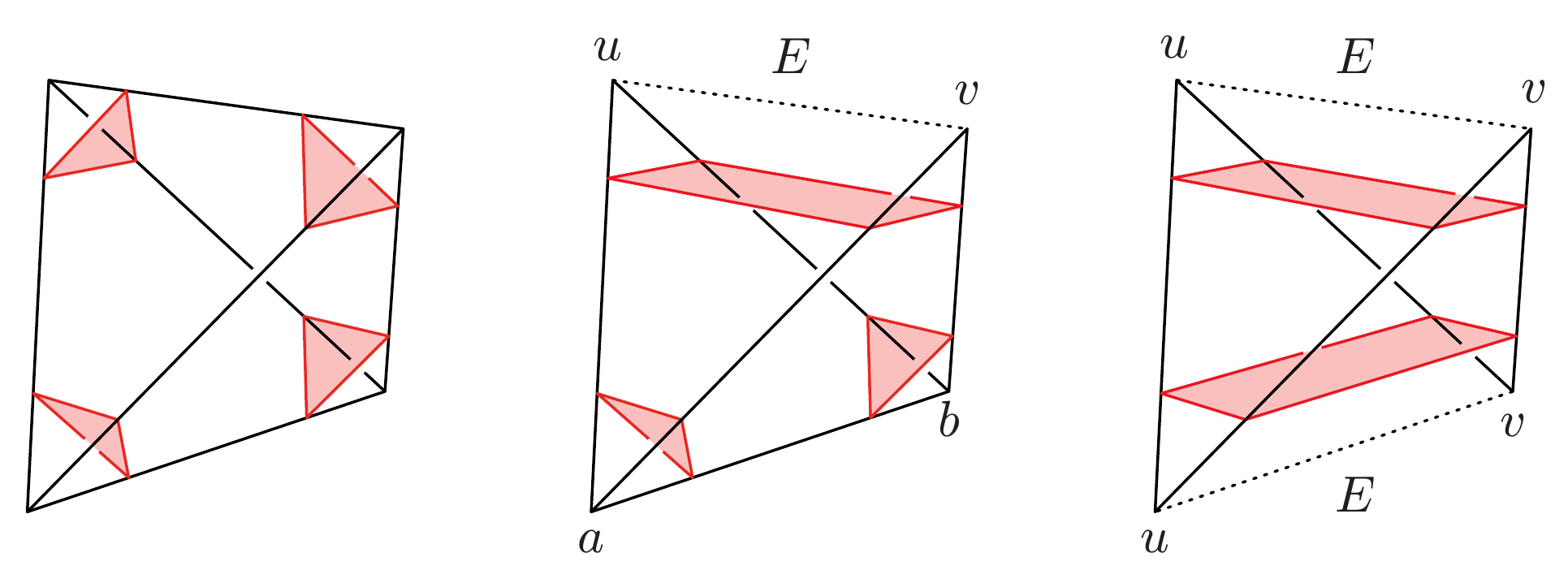}
\caption{The possible ways in which the edge $E$ can be incident to a tetrahedron.}
\label{possible_positions_of_E.pdf}
\end{figure}

\begin{figure}[htb]
\centering
\includegraphics[width=0.8\textwidth]{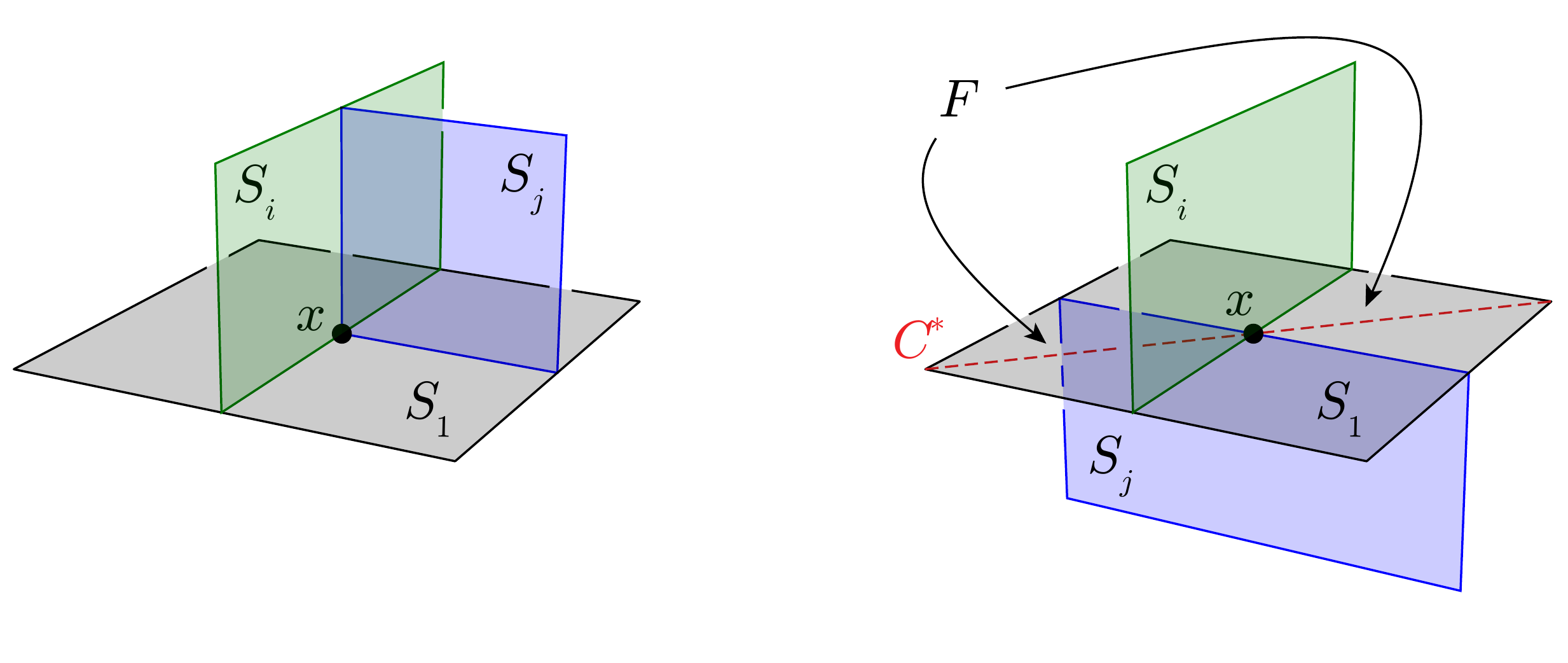}
\caption{The two possible arrangements of surfaces giving a tetrahedron incident with the two distinct vertices in $\mathcal{T}$.}
\label{hierarchy_arrangements.pdf}
\end{figure}

\subsubsection{} Figure~\ref{possible_positions_of_E.pdf} shows the possible ways in which the edge $E$ can be incident to a tetrahedron and the resulting normal discs in $S$: the edge $E$ either meets a tetrahedron only in the vertices, or in one edge, or in a pair of opposite edges. We will now show that the last case does not happen. Suppose there is a tetrahedron of $\mathcal T$ with two opposite edges identified to $E$. This tetrahedron is dual to a vertex $x$ of the hierarchy which is at the intersection of $S_1,S_i,S_j$, three surfaces of the hierarchy, with the first being the closed separating 2-sided surface. 
 If $S_i,S_j$ were on the same side of $S_1$, we would have $E$ dual to faces on different surfaces of the hierarchy, which cannot happen. Since $S_i,S_j$ must be on opposite sides of $S_1$, $x$ is a crossing point of curves $C_i \subset \partial S_i$ and $C_j \subset \partial S_j$ in $S_1$. The edge $E$ will then be dual to faces of the hierarchy on $S_1$ which have $x$ as a vertex and are diagonally opposite at $x$ (since the edges  of the tetrahedron corresponding to $E$ are opposite each other). Since each edge is dual to one face, we conclude that there is a face $F$ of the hierarchy with two vertices at $x$. Then there is a loop $C^*$ which is contained in $F$ and meets both $S_i$ and $S_j$ transversely only at $x$. We claim that $C^*$ is an element of infinite order in $H_1(M)$, contrary to assumption. To see this, first cut $M$ open along $S_1$, the surface containing the loop. We now use the Mayer-Vietoris sequence to compute $H_1(M)$ as the result of gluing back along $S_1$. We can push $C^*$ into both sides of the surface. By duality with $S_i$ on one side and $S_j$ on the other, we see that both copies of the loop are elements of infinite order in the homology of the cut open manifold $M_1$. Hence the Mayer-Vietoris sequence shows that gluing produces an element of infinite order coming from gluing the two copies of $C^*$ together. 
 So we conclude this case does not occur.
 
\subsubsection{} It follows that each tetrahedron contains $E$ at most once, and since the endpoints of $E$ are distinct vertices in the triangulation, we can crush $E$ and the surrounding tetrahedra to obtain a triangulation $\mathcal T^*$ as claimed.

\subsubsection{} To complete the proof of the theorem (showing that $\mathcal T^*$ is essential and strongly essential), we follow a similar argument to the case of infinite homology. First suppose that $\alpha$ is an edge loop in $\mathcal T^*$. Then there is an edge $\beta$ in $\mathcal T$ which becomes $\alpha$ after crushing. If $\beta$ is also an edge loop, then we can follow exactly the same argument as in the first case to get a contradiction. So assume that $\beta$ joins the two vertices and so $\beta \cup E$ is a loop. Clearly $\alpha$ being contractible is equivalent to $\beta \cup E$ being contractible. Since $\beta$ joins the two vertices, it must be dual to $S_1$, and so $\beta \cup E$ is contractible if and only if there is a homotopy between $\beta$ and $E$.
But now we can use the same argument as in case 1 to show that $E$ and $\beta$ are dual to the same face of $S_1$ and so are in fact the same edge. But then $\alpha$ cannot have come from $\beta = E$, and we have our contradiction. This finishes the first part showing that $\mathcal T^*$ is essential. Finally suppose that $\alpha_1,\alpha_2$ are homotopic edge loops in $\mathcal T^*$. Again choose edges $\beta_1,\beta_2$ which map to $\alpha_1,\alpha_2$ under the crushing map from $\mathcal T$ to $\mathcal T^*$. There are three cases: 
 
\subsubsection{} In the first case, $\beta_1, \beta_2$ have the same endpoints and are homotopic keeping their endpoints fixed.  As before, the same argument as in the case of infinite homology starting at \ref{pull off homotopy two surfaces} shows this cannot happen. (Note that this case covers both the endpoints of each edge being the same or different, and the argument is the same in each case.) 
 
\subsubsection{} In the second case $\beta_1$ has both endpoints on one of the two vertices, and $\beta_2$ has both endpoints on the other vertex, and
 the loop 
formed by following $\beta_1$ then $E$, then $\beta_2$ in reverse then $E$ in reverse,
 is contractible. In this case, both $\beta_1, \beta_2$ are loops. We can now follow the previous argument involving Mayer Vietoris to show that the homology class of the homotopic loops $\beta_1, \beta_2$ is of infinite order, since these loops are dual to spanning surfaces on either side of $S_1$. So this contradicts our assumption that the first homology is finite. 
  
\subsubsection{} In the third case, without loss of generality $\beta_2$ has both endpoints at the same vertex, while $\beta_1$ has them at different vertices, and the loop $\beta_1 \cup \beta_2 \cup E$ is contractible, giving us a homotopy $H\co D\rightarrow M$. Since it has different vertices at its endpoints, $\beta_1$ must be dual to $S_1$, and since it has the same vertex at its ends, $\beta_2$ is dual to some $S_i, i>1$. As before, our plan is to pull the homotopy off most of the hierarchy, while fixing the boundary. This time, our goal is that the intersection of $H(D)$ with the hierarchy consists of a `T' pattern: an arc of intersection with $S_1$ joining a point of $E$ with a point of $\beta_1$, together with an arc of intersection with $S_i$ joining a point of $\beta_2$ with a point of the first arc. Then, collapsing $E$ corresponds to deleting the subarc of the $S_1$ arc joining the triple point to the $E$ section of the boundary, then collapsing the $E$ part of the boundary. The result is a homotopy showing that $\alpha_1$ and $\alpha_2$ are dual to the same face of the hierarchy and so are in fact the same edge.

\subsubsection{} To achieve our goal `T' intersection pattern, first we pull $H(D)$ off all surfaces of the hierarchy from $S_1$ down to $S_{i-1}$ other than the single arc of intersection with $S_1$. Loops are removed using incompressibility as usual, and arcs of intersection that meet the intersection arc with $S_1$ are removed as in \ref{remove arcs bigons}. When we come to $S_i$, we find that one of our arcs of intersection hits $\beta_2$ at one end, and the $S_1$ arc at the other, since there is nowhere else that the arc can go to. We remove any other intersections of $S_i$ with $H(D)$ as usual. At this stage then we have our `T' intersection pattern, together with some intersections with surfaces from $S_{i+1}$ down to $S_k$.  Let $j \in [i+1,k]$ be the least integer such that   $S_j$
has non-empty intersection with the image of the homotopy $H$.
There are two cases, depending on whether an arc of $S_j$ has both endpoints on the same one of the three ``legs'' of the `T', or on different legs. The former case is dealt with innermost first, again as in \ref{remove arcs bigons}, and we perform all of these homotopies first. Now for the second case, there are two subcases depending on which one of the two components formed by the initial separating surface $S_1$ contains $S_j$. If $S_j$ is in the component not containing $S_i$ then the $S_j$ arc runs over the ``top'' of the `T', and no parts of $S_j$ can appear on the other side of $S_1$. This case is again dealt with as in \ref{remove arcs bigons}. 

\begin{figure}[htb]
\centering
\includegraphics[width=0.7\textwidth]{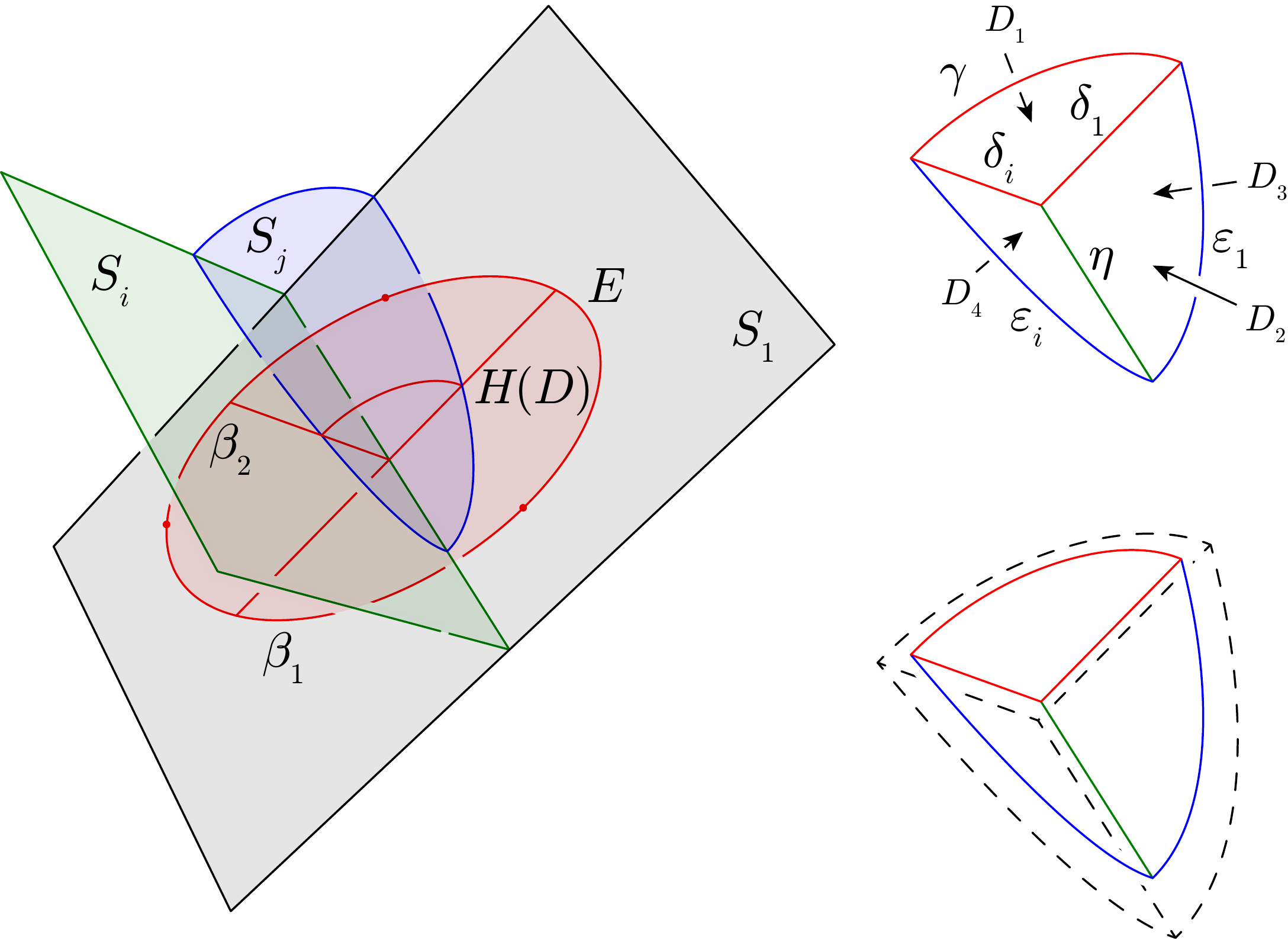}
\caption{The image of the homotopy disc intersecting the hierarchy. The diagrams to the right are a close-up view of the intersection arcs on the left, and the way to alter the homotopy to avoid the arc of intersection with $S_j$.}
\label{remove_arc_triangle_version.pdf}
\end{figure}

\subsubsection{}\label{T case bdry incomp} Finally, we deal with $S_j$ arcs that connect between an $S_1$ leg and the $S_i$ leg. We order these by the distance of the $S_i$ end of the arc from the triple point of the `T', and deal with the $S_j$ arc closest to the triple point first. See Figure \ref{remove_arc_triangle_version.pdf}. Here the arc of intersection with $S_j$ is labelled $\gamma$. Since $\gamma$ is the innermost arc, there are no other arcs of intersection inside the disc $D_1$. As in the situation in \ref{remove arcs bigons}, boundary incompressibility of $S_j$ in $M_{j-1}$ together with minimal lexicographical complexity means that there is a disk $D_2$ in $S_j$ bounded by $\gamma$ and another arc, here the union of $\epsilon_1 \subset S_1$ and $\epsilon_i \subset S_i$. This arc must be contained entirely in $S_1$ and $S_i$ (with one connected component in each) by minimal lexicographical complexity for $S_j$. The arcs $\epsilon_1$ and $\epsilon_i$ have no intersections with $S_{j'}$ for $i<j'<j$, for the same reasons.

\subsubsection{}\label{eta} Now consider $\eta$, the part of the boundary curve of $S_i$ between $\delta_i \cap \delta_1$ and $\epsilon_i \cap \epsilon_1$. This must be a single arc on $S_i$ which does not intersect any other surface boundary curve, by minimal complexity for $S_i$, comparing it with the curve $\delta_1\cup\epsilon_1$. The curves $\eta, \delta_i$ and $\epsilon_i$ together form a loop in $S_i$, and bound a disc $D_4$ in $S_i$ by the incompressibility of $S_i$. Here there is no question of which side of the loop contains the disc since $\eta$ is on the boundary of $S_i$. Also, the curves $\eta, \delta_1$ and $\epsilon_1$ together form a loop in $S_1$, and bound a disc $D_3$ in $S_1$ by the incompressibility of $S_1$. If the disc were ``outside'' the loop, so that $E$ and $\beta_1$ intersect it, then similarly to as in \ref{inside outside}, the union of $D_3$ with the discs $D_1, D_2$ and $D_4$ gives a sphere, so by irreducibility of $M$, one or the other side is a ball. But then either $S_i$ or $S_1$ would be contained in the ball (depending on which side of the sphere it is), which is impossible because of the incompressibility of the surfaces. Therefore $D_3$ is on the ``inside'' of the loop $\eta \cup \delta_1 \cup \epsilon_1$ as shown in Figure \ref{remove_arc_triangle_version.pdf}.

\subsubsection{} We claim that there are no curves of intersection between the interiors of any of $D_2, D_3, D_4$, and any of the boundary arcs of the hierarchy. First, $D_2 \subset S_j$, so none of the surfaces with index less than or equal to $j$ can have boundary on it. Second, $D_3$ has no loops of intersection in its interior since it is a disc. There are no intersections with arcs of $S_j$, since $S_j$ is contained in the other component of $M\setminus S_1$. For intersections with $S_{j'}$ for $j'<j$, $D_3$ has boundary $\delta_1, \eta$ and $\epsilon_1$. The arc $\delta_1$ has no intersections by our choice of $H(D)$. The arcs $\epsilon_1$ and $\eta$ have no intersections with $S_{j'}$ for $j'<j$ by our previous arguments in \ref{T case bdry incomp} and \ref{eta} respectively. So there are no intersections with $D_3$. Lastly, $D_4$ again has no loops of intersection in its interior since it is a disc. For intersections with $S_{j'}$ for $i<j'<j$, there are no intersections with $\delta_i$ by our choice of $H(D)$, and no intersections with $\epsilon_i$ or $\eta$ again by our arguments in \ref{T case bdry incomp} and \ref{eta} respectively. 
 For intersections of $S_j$ with $D_4$, the above arguments also rule out any arcs that do not intersect $\delta_i$. Any arc that does intersect $\delta_i$ implies that there is an arc of intersection in $H(D) \setminus D_1$ starting at $\delta_i$. However, we have already removed all such arcs that have both endpoints on the $S_i$ arc, and the only other option would be an arc connecting an $S_1$ arc and the $S_i$ arc that is closer to the triple point of the `T' than the supposedly innermost such arc we are considering. So there are no arcs of intersection with any of $D_2, D_3$ or $D_4$. Therefore we can modify the homotopy, eliminating the arc $\gamma$, by replacing $H$ on a small neighbourhood of the pullback of $D_1$ with a map to the union of $D_2, D_3$ and $D_4$, pushed slightly off as in the lower right diagram of Figure \ref{remove_arc_triangle_version.pdf}. Continuing in this way, we can remove all arcs of intersection between $S_j$ and $H(D)$. 

\subsubsection{} Again, we repeat the process for $S_{j''}$ where $j<j''\leq k$. 
Eventually, we end up with the graph of intersection between $H(D)$ and the hierarchy consisting of the desired `T' graph, and we are done with this case, and the proof of the main statement is complete.
 The last sentence of the theorem follows since a closed $P^2$-irreducible non-orientable 3--manifold is Haken. 
 \end{proof}

\begin{Rem}
It is also possible to extend this technique to cusped hyperbolic 3--manifolds as illustrated by the following examples. We will not pursue this but use a different technique to produce essential and strongly essential triangulations for cusped and closed hyperbolic $n$-manifolds in 
Section \ref{constant curvature}. In the cusped case, the technique is to subdivide the canonical decomposition into ideal hyperbolic polyhedra, using \cite{EP}, into ideal simplices. 
In the closed case, we follow a similar approach starting with a Delaunay cell decomposition.
\end{Rem}

\begin{figure}[htb]
\centering
\includegraphics[width=0.6\textwidth]{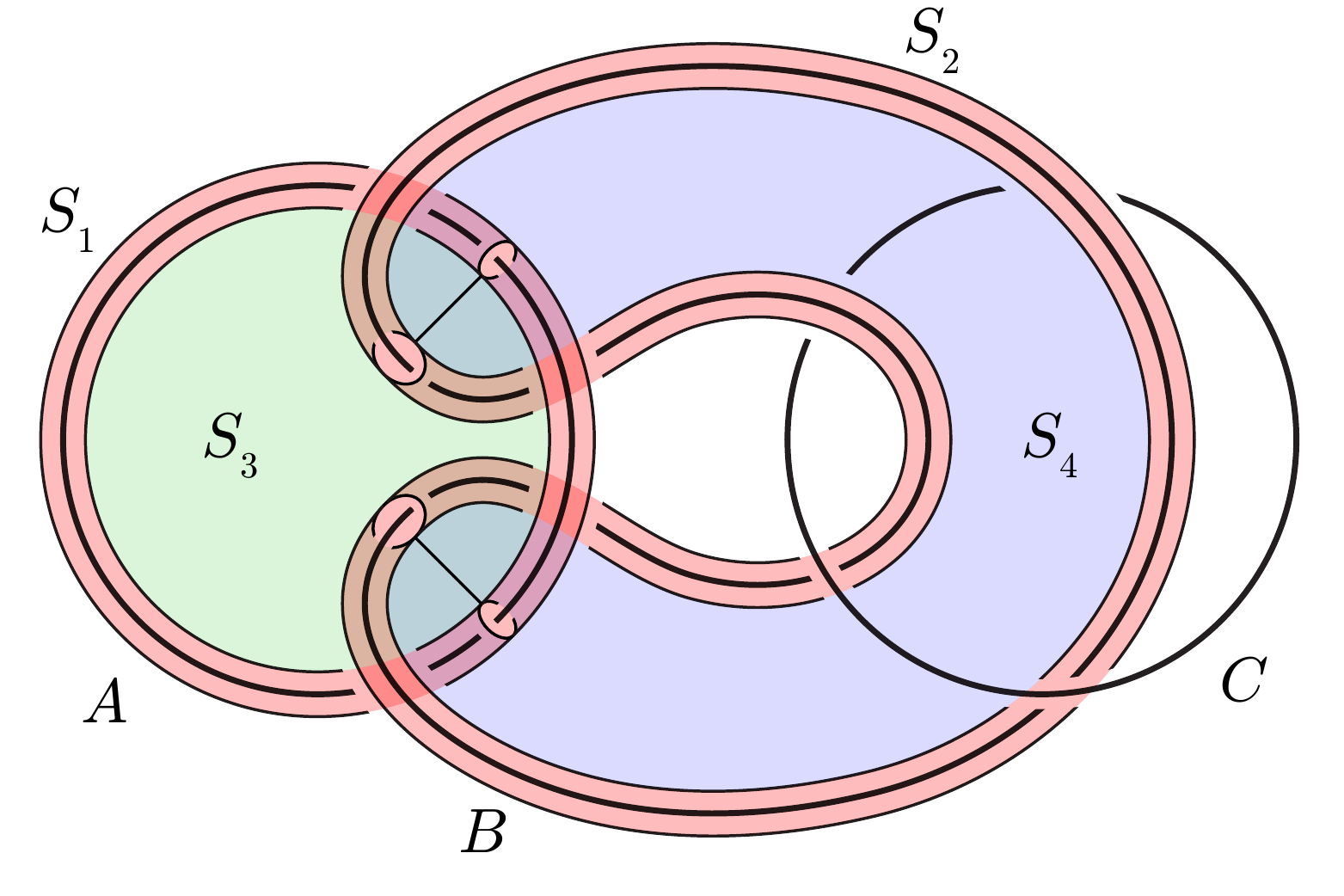}
\caption{A hierarchy for the Borromean rings.}
\label{borromean_hierarchy.pdf}
\end{figure}

\begin{Exa}[The Borromean rings]
See Figure \ref{borromean_hierarchy.pdf}. The three components of the Borromean rings are labelled $A, B$ and $C$. The first two surfaces, $S_1$ and $S_2$ 
\ifcolour
(shaded red)
\fi 
in the hierarchy are tori, parallel to $A$ and $B$. The third, $S_3$ 
\ifcolour
(shaded green)
\fi  
is a sphere with three holes, one boundary component is a longitude of $S_1$, and the other two are meridians of $S_2$. The fourth, $S_4$ 
\ifcolour
(shaded blue)
\fi 
is a disc, which can be described by starting with a sphere with three holes, this time with one boundary component a longitude of $S_2$ and the other two meridians of $S_1$, and then cutting this surface along the intersection with $S_3$. The result of this as shown in Figure \ref{borromean_hierarchy.pdf} is not a hierarchy, as the singular edges at the intersection between $S_3$ and $S_4$ are degree 4. We can fix this by altering $S_4$ slightly, pushing the components of $\partial S_4 \cap S_3$ off itself by moving it along $S_3$. There are four different ways to do this, as we can push each of the two pairs of components in two different directions. The cellulation dual to the complex we have before pushing off has square faces -- pushing off breaks each square into two triangles. The triangulations obtained by pushing off in one direction versus another are therefore related by flipping the diagonal of that square. This is sometimes referred to as a 4-4 move on the triangulation. 
\end{Exa}


\begin{figure}[htbp]
\centering
\includegraphics[width=0.6\textwidth]{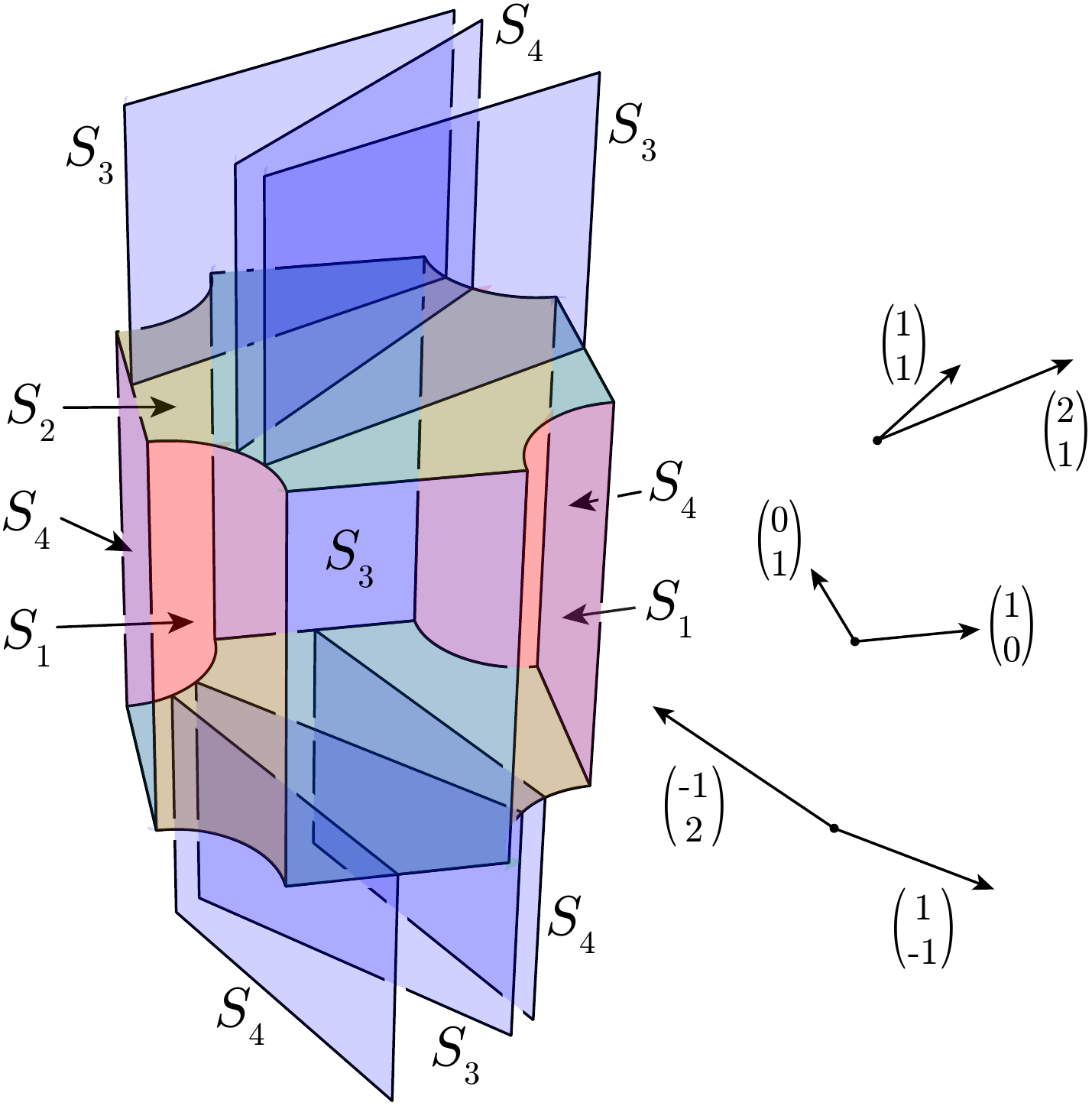}
\caption{A hierarchy for the figure 8 knot complement.}
\label{fig8_hierarchy.pdf}
\end{figure}

\begin{Exa}[The figure 8 knot complement] See Figure \ref{fig8_hierarchy.pdf}. We view the figure 8 knot complement as the punctured torus bundle with monodromy $$\left(\begin{array}{cc}2&1\\1&1\end{array}\right).$$ A fundamental domain for the manifold is a cube with its vertical edges deleted, the vertical faces of the cube identified by horizontal Euclidean translation, and the bottom face of the cube identified to the top via the monodromy. In the figure we see this cube, together with parts of surfaces in images of the cube under the monodromy and its inverse. The first surface of the hierarchy, $S_1$ \ifcolour
 (shaded red)
 \fi 
 is a boundary parallel torus. The second, $S_2$ \ifcolour
 (shaded green)
 \fi
 is the punctured torus fibre, cut off at $S_1$. The third and fourth, $S_3$ and $S_4$
 \ifcolour
 (shaded blue)
 \fi are vertical discs parallel to generators of the fundamental group of the punctured torus fibre.\\

\begin{figure}[htbp]
\centering
\includegraphics[width=1.0\textwidth]{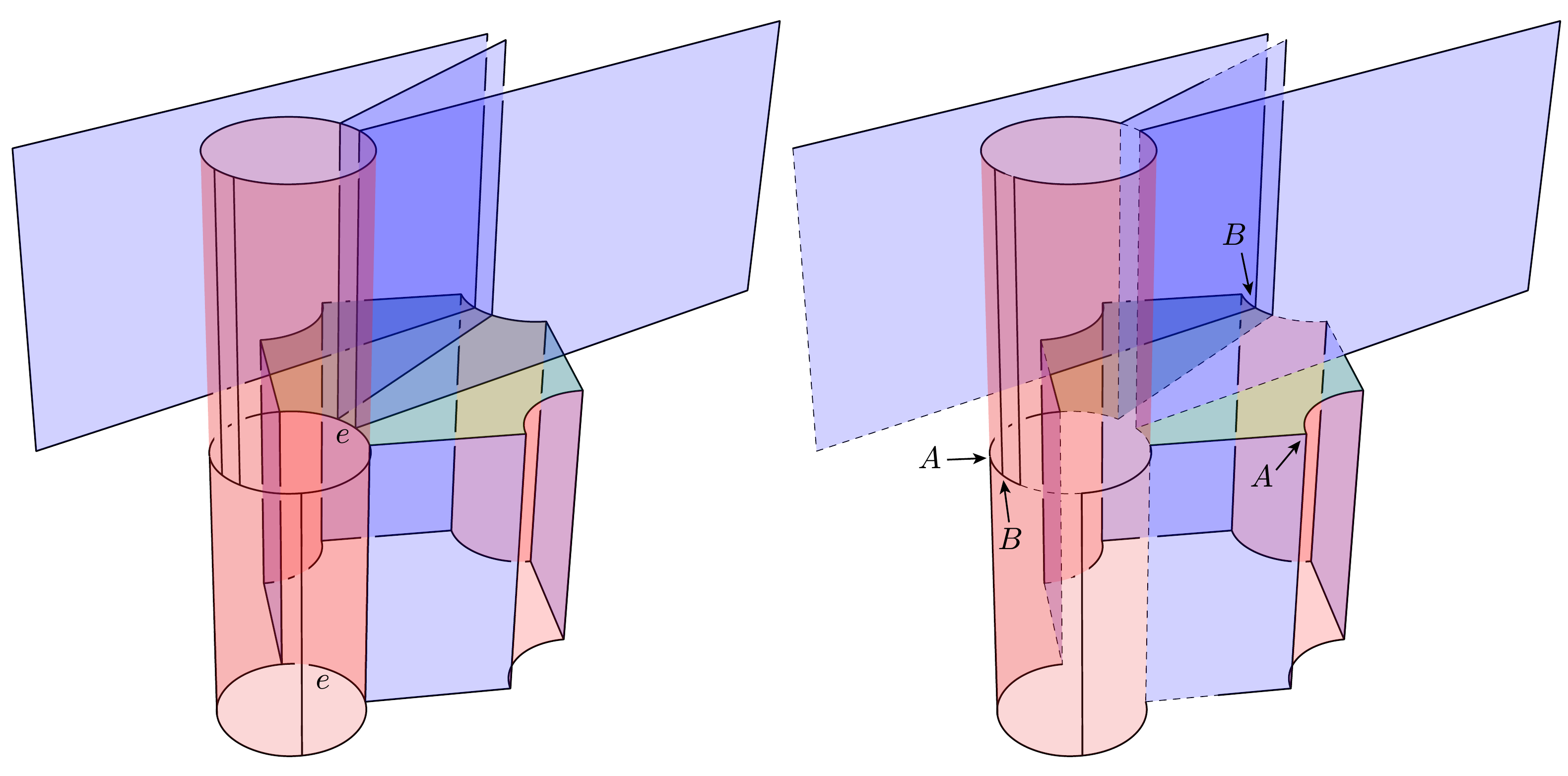}
\caption{Crushing an edge of the triangulation dual to the hierarchy corresponds to removing the faces of the hierarchy incident to an edge (here labelled $e$, note that we see two translates of it). After removing these faces, edges with only two faces incident are drawn with dashed lines.}
\label{fig8_hierarchy_crush.pdf}
\end{figure}

The triangulation dual to this hierarchy has one finite vertex and one ideal vertex. We crush an edge of this triangulation to produce an essential and strongly essential triangulation. See Figure \ref{fig8_hierarchy_crush.pdf}. On the left we see the hierarchy again. We need to crush an edge joining the finite and ideal vertices, or equivalently, remove a face of $S_1$. One choice is to remove one of the faces of $S_1$ incident to the edge marked $e$. However, the same face is incident to $e$ twice, so we must remove a face of $S_2$ as well, otherwise the resulting hierarchy would have a face with boundary. On the right, after removing these faces, many of the remaining faces coalesce into larger faces as there are no longer curves of triple points between them. In the dual picture, this corresponds to edges of the triangulation being identified after the crushing. There are only two quadruple points remaining, labelled $A$ and $B$. Thus the corresponding triangulation has only two tetrahedra, and in this case we get the canonical ideal triangulation of the figure 8 knot complement. 
\end{Exa}


 \section{Angle structures}

We give a quick summary of strict angle structures and semi-angle structures. A detailed discussion may be found in e.g. \cite{luo_tillmann_2008}. Given an ideal triangulation, each ideal tetrahedron has a real number, called an \emph{angle}, associated with each of its edges, so that opposite edges have the same angles. The angle equations are then that the sum of the three angles at an ideal vertex for each such tetrahedron is $\pi$ and the sum of the angles around each edge in the triangulation is $2\pi$. A strict angle structure is a solution of the angle equations for which each angle is strictly positive, whereas a semi-angle structure is a solution with each angle non-negative. 
 
 \begin{Thm}\label{sas implies strongly essential}
 Suppose that $M$ is a 3--manifold with an ideal triangulation $\mathcal T.$ 
  If $\mathcal T$ admits a strict angle structure, then $\mathcal T$ is strongly essential.
 If $\mathcal T$ admits a semi-angle structure, then it is essential. 
  \end{Thm}

\begin{proof}

Consider the case of a strict angle structure on $\mathcal T$. We establish first that $\mathcal T$ is essential. So suppose not, i.e.\thinspace there is an edge $\alpha$ of $\mathcal T$ which is homotopic, keeping its ends on a boundary torus $T$, into $T$. Construct the covering $\overline M$ corresponding to the image of the subgroup $\pi_1(T)$ in $\pi_1(M)$. Then there are homeomorphic lifts $\overline T$ of $T$ and $\overline \alpha$ of $\alpha$ in $\overline M$, so that ends of $\overline \alpha$ lie in $\overline T$. There is a collar of the boundary component $\overline T$ of $\overline M$ which contains $\overline \alpha$. We can push $\overline T$ across $\overline \alpha$ to give a new surface $\overline T^\prime$ that is possibly not normal. Using the barrier arguments of \cite{0-efficient}, the edge $\overline \alpha$ is a barrier for normalising the incompressible surface $\overline T^\prime,$ and hence we shrink $\overline T^\prime$ to obtain a normal surface $T^*$ in $\overline M$, which is isotopic but not normally isotopic to $\overline T$ since $T^*$ is disjoint from $\overline \alpha.$

According to the discussion of the result of Casson and Rivin in \cite{lackenby00} (see also \cite{luo_tillmann_2008}, Theorem 3), there cannot be any normal torus $T^*$ which is not peripheral in a triangulation with a strict angle structure. A direct argument is as follows. Note that $\overline M$ has a lifted strict angle structure. We apply a formal Gauss-Bonnet argument (cf.\thinspace \cite{luo_tillmann_2008}, Proposition 13).  Namely if we put the dihedral angles of edges at the vertices of the normal cells of $T^*$, then a triangle has angle sum $\pi$ and a quadrilateral has angle sum $2\alpha +2\beta<2\pi$, where $\alpha,\beta,\gamma$ are the 3 angles at edges of the ideal tetrahedron containing the quadrilateral. Adding zero for each triangle and $2\alpha+2\beta-2\pi=-2\gamma$ for each quadrilateral gives $2\pi \chi (T^*)$. We conclude that the torus $T^*$ has negative Euler characteristic if it has any quadrilateral disks. But since $T^*$ is not normally isotopic to a peripheral torus in the ideal triangulation, it must have quadrilaterals. So this gives a contradiction to the existence of the homotopy of an edge into a boundary torus. This completes the proof that $\mathcal T$  is essential.

To show that $\mathcal T$ is strongly essential, assume to the contrary that there are edges $\alpha_1,\alpha_2$ which are homotopic keeping their ends on peripheral tori. We can truncate $M$ and $\mathcal T,$ and finish the proof with the following argument of Lackenby \cite{lackenby00}. Consider the homotopy as a mapped in quadrilateral disc with two opposite boundary arcs running along peripheral tori and two arcs running along truncated copies of $\alpha_1,\alpha_2$. Making this map transverse to the truncated triangulation and least weight, we can classify all the possible disc types in the complement of the pullback of the $2$-skeleton of $\mathcal T$. 
Lackenby shows that all the disc types have non-positive formal contribution to the Euler characteristic in a combinatorial Gauss--Bonnet formula. 
Moreover there are some negative contributions so this contradicts the fact that the Euler characteristic is one. So this completes the proof that the triangulation is strongly essential if it supports a strict angle structure.

\begin{figure}[htb]
\centering
\includegraphics[width=0.4\textwidth]{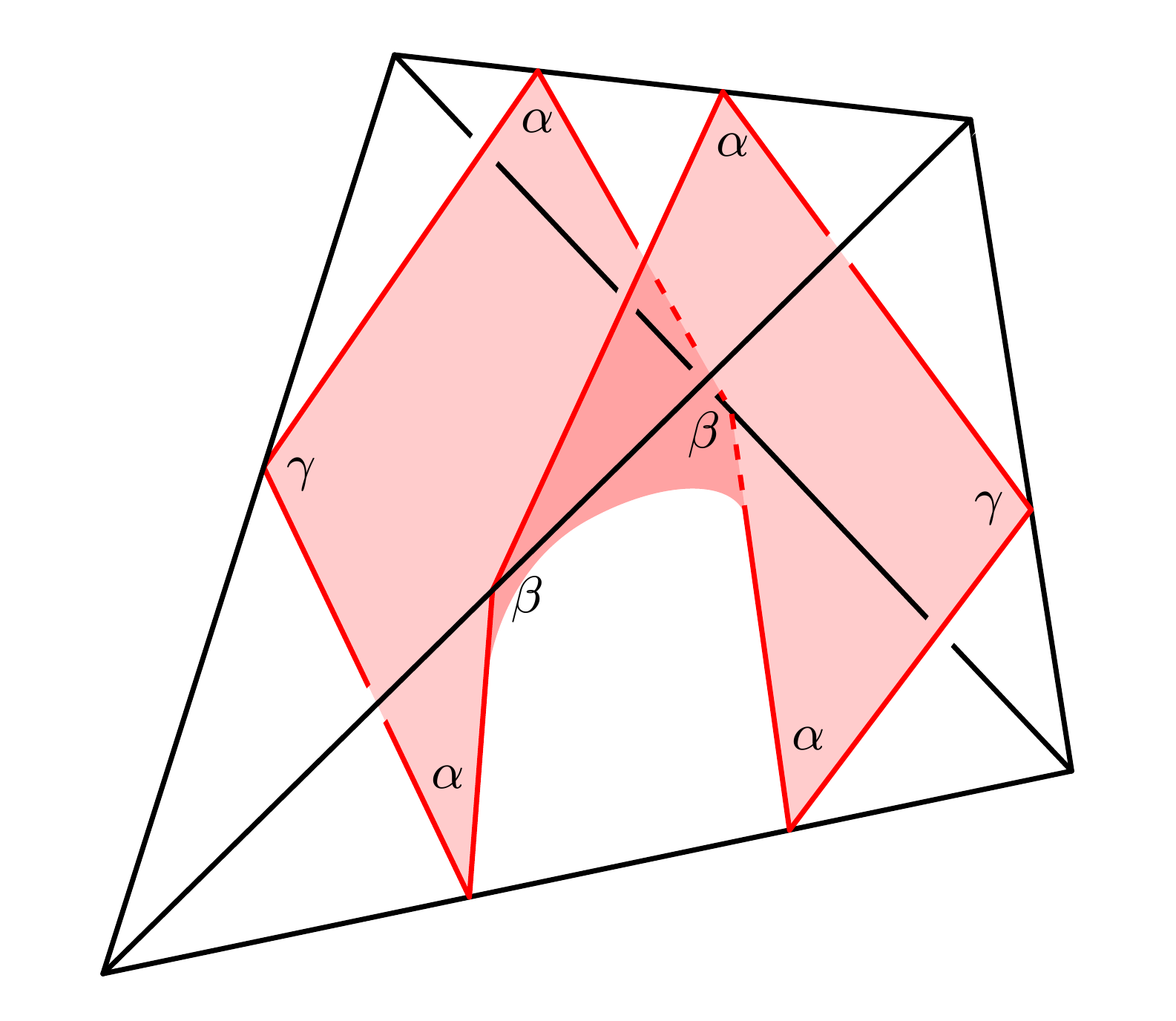}
\caption{An octagonal disc from an almost normal surface meeting the edge with angle $\alpha$ twice contributes negative combinatorial area $2(\alpha - 2\pi)$ to the combinatorial Gauss-Bonnet formula.}
\label{normal_octagon.pdf}
\end{figure}

Suppose next that $\mathcal T$ has a semi-angle structure. As in the first paragraph, if $\mathcal T$ was not essential, there would be a normal torus $T^*$ embedded in the covering space $\overline M$ which has a lifted taut triangulation. A standard argument using thin position (which goes back to Thompson~\cite{thompson-1994} and Stocking~\cite{stocking-2000}, with a detailed discussion given in Schleimer~\cite{schleimer-2001}, Theorem 6.2.2) shows that in the product region between the vertex linking torus and the normal torus $T^*$ is an almost normal torus $T^\prime$ which has exactly one octagon. However, this octagon makes a negative contribution to the formal Euler characteristic of $T^\prime$ (see Figure~\ref{normal_octagon.pdf}), thus contradicting the combinatorial Gauss-Bonnet formula.
This completes the proof of this case. 
\end{proof}


\begin{Exa}
We first give simple examples of taut triangulations which are not strongly essential.
Suppose we take any example of an ideal triangulation with a taut structure and choose a pair of adjacent triangular faces, sharing an edge $e$.   Assume also that the $\pi$ angles at $e$ are on opposite sides of the two faces.  Unglue the triangulation along $e$ and these two faces, and glue onto the resulting four faces a pillow consisting of two tetrahedra with two faces and a degree two edge in common (sometimes this is called a 0--2 move).
It is then easy to see that the taut structure extends over this new triangulation. 
However the degree two edge is an obstruction to the existence of a strict angle structure (see \cite{kang_rubinstein_2005}). 
Moreover this triangulation clearly has two edges which are homotopic keeping their ends on the boundary tori, namely the two edges that the edge $e$ is split into. So we see explicitly that this triangulation is not strongly essential. 
\end{Exa}

\begin{Exa}[m136]
We now give an example of a taut triangulation which is strongly essential, but does not support a strict angle structure. This is
Example 7.7 of \cite{HRS}, which is a triangulation of the manifold m136 from the SnapPea \cite{SnapPy} census. For convenience we reproduce the combinatorial data of this example in Tables~\ref{table_example_gluings} and \ref{table_example_edges}.

To reconstruct the triangulation, take 7 tetrahedra with vertices labelled 0 through 3, all with consistent orientations, and make the appropriate gluings given in Table~\ref{table_example_gluings}. For example, the top left entry in the table tells us to glue the face of tetrahedron \#0 with vertices labelled 0,1,2 to the face of tetrahedron \#1 with vertices labelled 3,1,2 in the orientation that matches the vertices in the order given. The identifications of edges can be deduced from this information, and the resulting edge cycles are given in Table~\ref{table_example_edges}.

\begin{table}[htb]
\begin{center}
{\small
\begin{tabular}{ |r|cccc|c| }   
\hline
Tetrahedron  & Face 012 & Face 013 & Face 023 & Face 123 & Shape parameter at 01\\ 
\hline
0 & 1 (312) & 4 (302) & 6 (130) & 4 (132) & $\hspace{0.9cm} 2i$\\
1 & 3 (102) & 2 (012) & 2 (203) & 0 (120) & $ -1+2i$\\
2 & 1 (013) & 6 (321) & 1 (203) & 4 (031) & $\hspace{0.3cm}\frac{3}{5} + \frac{1}{5}i$\\
3 & 1 (102) & 6 (230) & 5 (021) & 5 (023) & $-1\hspace{0.6cm}$\\
4 & 5 (312) & 2 (132) & 0 (130) & 0 (132) & $\hspace{0.3cm}\frac{1}{5} + \frac{2}{5} i$\\
5 & 3 (032) & 6 (012) & 3 (123) & 4 (120) & $2\hspace{0.4cm}$\\
6 & 5 (013) & 0 (302) & 3 (301) & 2 (310) & $\hspace{0.3cm}\frac{1}{2} + \frac{1}{2}  i$\\ 
\hline
\end{tabular}
\vspace{0.3cm}
}
\caption{Gluing data for a triangulation of m136}
\label{table_example_gluings}
\end{center}
\end{table}

\begin{table}[htb]
\begin{center}
{\small
\begin{tabular}{ |r|r|l| }   
\hline
Edge  & Degree &  Tetrahedron (Vertex numbers)\\ 
\hline
0 & 4 & 0 (01), 4 (30), 2 (21), 1 (31) \\
1 & 4 & 0 (02), 1 (32), 2 (30), 6 (13) \\
2 & 10 & 0 (03), 6 (10), 5 (10), 3 (30), 6 (02), 5 (03), 3 (13), 6 (30), 0 (23), 4 (32)\\
3 & 10 & 0 (12), 4 (13), 2 (32), 1 (30), 2 (20), 1 (02), 3 (12), 5 (02), 3 (02), 1 (12)\\
4 & 6 & 0 (13), 4 (02), 5 (32), 3 (32), 5 (12), 4 (12)\\
5 & 4 & 1 (01), 2 (01), 6 (32), 3 (10)\\
6 & 4 & 2 (13), 6 (21), 5 (31), 4 (01)\\
\hline
\end{tabular}
\vspace{0.3cm}
}
\caption{Edge cycles in the triangulation}
\label{table_example_edges}
\end{center}
\end{table}

It is shown in \cite{HRS} that the triangulation does not admit a strict angle structure. However, it can be checked (for instance, using Regina \cite{Regina}) that the triangulation does admit a taut angle structure. We now show that the triangulation is strongly essential. This shows that unfortunately, having a strongly essential triangulation is necessary but not sufficient to produce strict angle structures, even starting with a taut structure.

Although the property of a triangulation being strongly essential is a homotopic property, here we use the geometric structure. The last column of Table \ref{table_example_gluings} gives shape parameters for the seven tetrahedra that form a solution to Thurston's gluing equations corresponding to the complete hyperbolic structure. The corresponding developing map $D\co \widetilde{M} \to \H^3$ has the property that distinct cusps of $\widetilde{M}$ map to distinct points of $\partial \H^3$. If there are two homotopic edges in $\widetilde{M}$, then the corresponding geodesics in $\H^3$ will coincide. If we had a triangulation with all positive volume tetrahedron shapes corresponding to the complete structure (so a geometric triangulation), then we would be done. This follows because the geodesic image $g$ of an edge in $\H^3$ is surrounded by positive volume tetrahedra, and no other tetrahedra in the developing map can intersect these other than via the face identifications with their neighbours. In particular, no other tetrahedron in $\H^3$ can have $g$ as an edge.  (Alternatively, a geometric triangulation gives a strict angle structure, so we can apply Theorem \ref{sas implies strongly essential}.) 

For the shapes that we do have, each positive volume tetrahedron in $\H^3$ again intersects other tetrahedra only via face identifications. So if there were a pair of homotopic edges, we would be able to walk from one to the other through a sequence of flat tetrahedra. In our case the flat tetrahedra are tetrahedra 3 and 5. Figure \ref{m136_ex_flat_tetrahedra.pdf} shows the way in which the two tetrahedra are glued to each other. Notice that they are glued to each other on the bottom two faces of tetrahedron 3 and the top two faces of tetrahedron 5. Moreover, as we develop through copies of these tetrahedra in $\widetilde{M}$, each new pair of tetrahedra are parabolic translates of the previous pair (around the vertex labelled 2 in both tetrahedra). Thus, no pair of edges are mapped to the same geodesic in $\H^3$.

We use a similar argument in the proof of Theorem \ref{cusped hyperbolic implies strongly essential}. 

\begin{figure}[htb]
\centering
\includegraphics[width=0.5\textwidth]{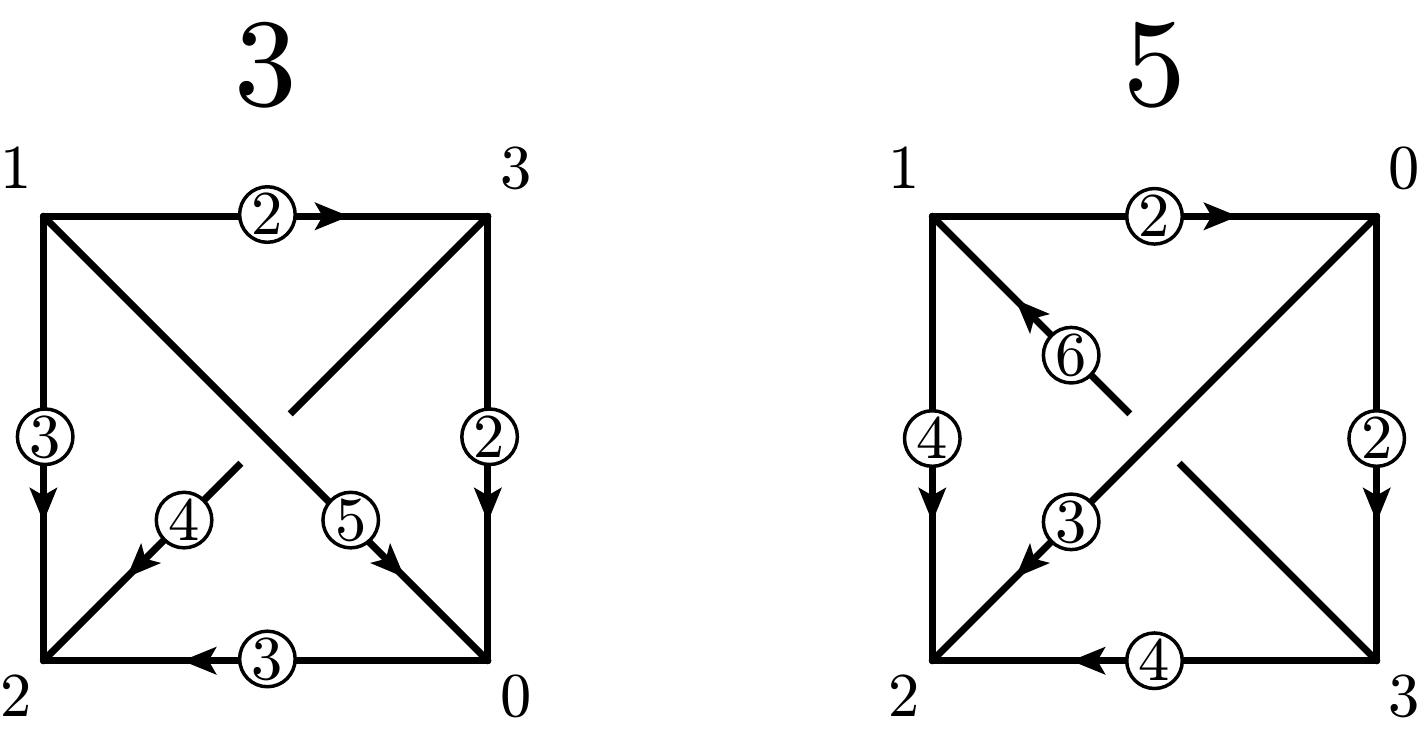}
\caption{The two flat tetrahedra in the triangulation with shapes corresponding to the complete structure. Edges are numbered in circles with arrows indicating the orientation of each edge relative to the tetrahedra. The numbering and orientations match the ordering given in Table \ref{table_example_edges}. Vertex numbers are also shown. The flat shapes for these tetrahedra are such that the diagonal edges have dihedral angles of $\pi$.}
\label{m136_ex_flat_tetrahedra.pdf}
\end{figure}
\end{Exa}


 \section{Metrics with constant curvature}\label{constant curvature}
 
 \begin{Thm}\label{cusped hyperbolic implies strongly essential}
Assume that $M$ is a compact orientable hyperbolic 3--manifold $M$ with $h>0$ incompressible torus boundary components. 
Then $M$ has a strongly essential ideal triangulation. 
 
\end{Thm}
 
\begin{proof}
This follows relatively easily from techniques in \cite{HRS}. Given our manifold $M$, we construct an ideal triangulation using Algorithm 5.2 of \cite{HRS}. Briefly, this uses the Epstein-Penner decomposition~\cite{EP} of $M$ into convex ideal hyperbolic polyhedra. These polyhedra are further subdivided into ideal pyramids using a coning procedure. For our purposes, the cone vertex on each polyhedron can be chosen arbitrarily. The pyramids are further subdivided into ideal tetrahedra. Where the induced triangulations on the glued faces of adjacent polyhedra disagree, we insert layered triangulations to bridge between them. This construction produces an ideal triangulation of $M$ that comes with ideal hyperbolic shapes for all of the tetrahedra, those coming from the polyhedral cells having positive volume while the tetrahedra in the bridging layered triangulations have volume zero. Thus there is a pseudo-developing map $D\co \widetilde{M}\to\H^3$ corresponding to the complete hyperbolic structure on $M$. The edges of the triangulation all map to geodesics in $\H^3$. Since a geodesic must have distinct endpoints in $\H^3$, this means that each edge in $\widetilde{M}$ must also have distinct endpoints, and so none of them are null-homotopic, and so the edges are all essential. 

A similar argument shows that two edges are homotopic fixing their endpoints if and only if they map to the same geodesic in $\H^3$. Our pseudo-developing map shows that this could only possibly happen within one of the layered triangulation regions. However, an analysis of the construction of the triangulations in these regions (Algorithm 5.1 of \cite{HRS}) shows that no pair of edges in such a region join the same vertices of the associated ideal polygon. This shows that the triangulation is also strongly essential.
\end{proof} 
 
 \begin{Rem}
Theorem~\ref{sas implies strongly essential} gives a slightly different construction proving Theorem~\ref{cusped hyperbolic implies strongly essential} in the case of a hyperbolic link complement in $S^3$. Corollary~1.2 of \cite{HRS} states that such a manifold admits an ideal triangulation with a strict angle structure. Then by Theorem \ref{sas implies strongly essential}, this triangulation is strongly essential.
\end{Rem}


Given a closed Riemannian $3$--manifold with constant curvature, our next aim is to mimic the proof of Theorem \ref{cusped hyperbolic implies strongly essential}, starting with the Delaunay cell decomposition dual to  the cut locus of a point, to obtain a strongly essential triangulation with a single vertex.
Without loss of generality, we may normalise the metric to have constant curvature $-1,$ $0$ or $+1.$ An extra condition on the diameter is required in the case of constant positive curvature.  

For constant zero or negative curvature, the cut locus from any point is always a spine, obtained from the boundary of a Dirichlet domain, but is not always generic so its dual need not be a triangulation. As an example, the flat metric on the torus obtained by gluing up a rectangle gives a non-generic spine but gluing up a centrally symmetric hexagon gives a generic spine. For the cut locus in the hexagon case is a theta curve and the dual is the usual essential one-vertex triangulation of a torus. 

\begin{Thm}\label{Dirichlet}
Suppose that $M$ is a closed Riemannian manifold of dimension $3$ with a metric with constant curvature $-1,$ $0$ or $+1.$ In the case of curvature $+1$,  also assume that the diameter of the manifold is smaller than  $\pi$.  Then the cell decomposition dual
to the cut locus of any point $x_0$ can be subdivided to give  a strongly essential one-vertex triangulation. 
\end{Thm}

\begin{proof}
We first follow the well-known construction of Dirichlet domains and Delaunay cell decompositions in the case of constant zero or negative curvature. The modification required to deal with the case of constant positive curvature is given at the end. For basic information on hyperbolic manifolds, see \cite{rat}. 

Consider the picture in the universal covering $\widetilde{M}$ of $M$. Since $\widetilde{M}$ is isometric to Euclidean or hyperbolic space, there is a unique geodesic segment between any two points of $\widetilde{M}$.
The Dirichlet construction defines a region $\mathcal D$ in $\widetilde{M}$ given by choosing a lift $x_1$ of the base point $x_0$ in $M$ and letting ${\mathcal D} = \{x \in \widetilde{M}: d(x,x_1) \le d(x,x_j) \text{ for all } j >1 \}$, where the lifts of $x_0$ are denoted $x_1,x_2, \dots$ and $d$ is the Riemannian distance between points in $\widetilde{M}$.  Then the Dirichlet domain $\mathcal D$ is a convex geometric polyhedron with totally geodesic faces, and the boundary of $\mathcal D$ projects to the cut locus of  $x_0$ in $M$, which gives a spine for $M$. 

Now $\widetilde{M}$ is tiled by copies of the fundamental domain $\mathcal D$ under the action of the covering transformation group. 
The geometric dual to this tiling is the Delaunay cell decomposition $\mathcal C$ of $\widetilde M$ into convex polyhedra with vertices at the points $x_1,x_2, \dots$. In particular, each edge of $\mathcal C$ is a geodesic in $\widetilde M$ joining two distinct lifts of $x_0$. 
Then $\mathcal C$ projects to a decomposition of $M$ into convex geometric polyhedra with a single vertex $x_0$.
We can now repeat the argument from the proof of Theorem \ref{cusped hyperbolic implies strongly essential} to obtain a subdivision into topological  tetrahedra giving a strongly essential triangulation.

In the case of constant positive curvature $+1$, if the diameter is strictly less than $\pi$, then the Dirichlet domain lies in a ball of radius at most $\frac{\pi}{2}$. In this case, the Dirichlet domain is strictly convex and so the same argument works as in the constant zero or negative curvature cases. 
\end{proof}

\begin{Rem}
It seems likely that both Theorem \ref{cusped hyperbolic implies strongly essential} and Theorem  \ref{Dirichlet} should also apply in higher dimensions. If there are no coincidences of geometry then the Epstein-Penner decomposition consists only of simplices of the appropriate dimension, and then, by the arguments above, this triangulation is strongly essential. In the non-generic case we would need to subdivide convex ideal hyperbolic polytopes into ideal simplices and bridge between incompatible induced triangulations of the polytopes' facets. The former is easy enough, either using coning constructions or regular triangulations (as used in \cite{GHRS}, see also \cite{GKZ}), but it isn't immediately obvious how to do the latter by layering simplices. 

Similar observations apply to the cut locus and its dual cell decomposition. 

\end{Rem}

\begin{Exa}[Lens spaces]
It is shown by Anisov (\cite{Anisov-2006}, Theorems 2.4 and 3.1) and Gu\'eritaud (\cite{guer-2009}, proof of Theorem 3) that we get the minimal layered triangulation of a lens space from the above construction, i.e.\thinspace the layered triangulations with the minimal number of tetrahedra (see \cite{JR-layered}).
In this case, the cut locus of a generic point is the spine dual to the minimal layered triangulation of $L(p,q),$ unless $q = \pm 1 \mod p,$ in which case one needs to perturb the cut locus slightly. 
\end{Exa}

\begin{Rem}
Ehrlich-Im Hof \cite[p650]{EhrlichImhof} mention examples of Riemannian metrics with no conjugate points on manifolds of dimension $\ge 3$  for which the dual triangulation to the cut locus is not strongly essential. 
\end{Rem}


\section{Algorithms}

The previous sections have given various constructions to obtain essential or strongly essential triangulations. We now show that there are algorithms to decide whether a given triangulation of a closed 3--manifold is essential or strongly essential, based on the fact that the word problem is uniformly solvable in the class of fundamental groups of compact, connected 3--manifolds. Similarly, we show that there is an algorithm to test whether an ideal triangulation is essential based on the fact that the subgroup membership problem is solvable in the same generality. 
However, to test whether an ideal triangulation is strongly essential, we require the double coset membership problem to be solvable for pairs of peripheral subgroups---here, we restrict the discussion to the most interesting classes of manifolds.

We address the closed case in \S\ref{sec:algo:closed} and the ideal case in \S\ref{sec:algo:ideal}. In each, we first determine group theoretic decision problems that are equivalent to deciding whether a triangulation is essential or strongly essential, and then determine conditions on the manifold that ensure that the problems can be solved. The reader is referred to the excellent survey of decision problems for 3--manifold groups by Aschenbrenner, Friedl and Wilton~\cite{AFW} for definitions and further references for the results we use.


\subsection{The closed case}
\label{sec:algo:closed}

First suppose that $\widehat{M}$ is a closed, connected 3--manifold with 1--vertex triangulation, and write $M = \widehat{M}.$ Then a finite presentation for $\pi_1(M, v),$ where $v$ is the vertex, is obtained from the triangulation by arbitrarily orienting the edges. Every generator of the presentation corresponds to an (oriented) edge loop $\gamma$, and every relator corresponds to a face. The triangulation is essential if and only if for each edge loop $\gamma,$ the corresponding element in the fundamental group is non-trivial, written $[\gamma] \neq 1 \in \pi_1(M, v).$ It therefore suffices to solve the word problem for each generator.

The triangulation is strongly essential if, in addition, for any two (oriented) edge loops $\gamma$ and $\delta,$ we have $[\gamma] \neq [\delta]$ and $[\gamma] \neq [\delta]^{-1}.$ This, again, reduces to solving the word problem.

We conclude that there is an algorithm to test whether a triangulation of $M$ is essential or strongly essential if there is a solution to the word problem in $\pi_1(M).$

The Geometrisation Theorem of Thurston and Perelman together with work of Hempel \cite{JH} implies that 3--manifold groups are residually finite. The word problem in the larger class of finitely presented, residually finite groups has been known to be uniformly solvable since the 1960s due to work of Dyson \cite{Dy64} and Mostowski \cite{Mos66}. The existence of the desired algorithms therefore follows in full generality from the following:

\begin{Thm}
The Word Problem is uniformly solvable in the class of fundamental groups of compact, connected 3--manifolds.
\end{Thm}

\begin{Coro}
Given a 1--vertex triangulation of a closed, connected 3--manifold, there is an algorithm to test whether it is essential, and if so, whether it is strongly essential.
\end{Coro}


\subsection{The ideal case}
\label{sec:algo:ideal}

The main ingredient to test whether an ideal triangulation is essential is the following:

\begin{Thm}[Friedl and Wilton \cite{FW}]
The Membership Problem is uniformly solvable in the class of fundamental groups of compact, connected 3--manifolds.
\end{Thm}

\begin{Coro}\label{cor:algo ideal is essential}
Given an ideal triangulation of the interior of a compact, connected 3--manifold, there is an algorithm to test whether it is essential.
\end{Coro}

\begin{proof}
We are given $M = \widehat{M}\setminus \widehat{M}^{(0)},$ where $\widehat{M}$ is a triangulated 3--dimensional closed pseudo-manifold (with an arbitrary number of vertices). Taking the second barycentric subdivision gives the compact core $M^c$ as described in Section~\ref{sec:Preliminaries}, and $M^c$ has a triangulation arising from the second barycentric subdivision of $\widehat{M}^{(0)}.$

Given any edge path $\gamma$ of $\widehat{M}$ with the same initial and terminal vertex, take the intersection of $\gamma$ with $M^c$ and connect the endpoints $a$ and $b$ of $\gamma \cap \partial M^c$ by an arc contained in the 1--skeleton of $\partial M^c$ to give a loop $\gamma^c$ in $M^c$ based at $a$ (this can be done in an algorithmic fashion). The triangulation of $M^c$ yields a finite presentation of $\pi_1(M^c, a)$ as well as a finite generating set for $H = \im(\pi_1(B, a) \to \pi_1(M^c, a))$, where $B$ is the component of $\partial M^c$ containing $a$. One now needs to test whether $[\gamma^c] \in H.$ The ideal triangulation of $M$ is essential if and only if the answer to this membership test is negative for each edge $\gamma$ of $\widehat{M}$ with the same initial and terminal vertex.
\end{proof}

Our test whether a triangulation is strongly essential requires a solution to the so-called \emph{Double Coset Membership Problem}. A class of groups for which this can be solved is provided in the discussion after Corollary 4.16 in \cite{AFW}. There it is argued that given the fundamental group of a compact Seifert fibred space or a topologically finite hyperbolic 3--manifold, the product of any two finitely generated subgroups is separable. In the case of Seifert fibred manifolds, this is due to Niblo~\cite{Ni92} building on work of Scott~\cite{Sc78}. In particular, the following result holds by applying the argument of the proof of Lemma 4.3 in \cite{AFW}.

\begin{Prop}\label{pro:Double Coset Membership Problem}
The Double Coset Membership Problem is uniformly solvable for pairs of finitely generated subgroups in the class of fundamental groups of compact Seifert fibred spaces or topologically finite hyperbolic 3--manifolds.
\end{Prop}

\begin{Coro}\label{cor:algo ideal of geometric is strongly essential}
Given an ideal triangulation of a topologically finite hyperbolic 3--manifold or the interior of a compact Seifert fibred space with boundary, there is an algorithm to test whether it is strongly essential.
\end{Coro}

\begin{proof}
We follow verbatim the proof of Corollary~\ref{cor:algo ideal is essential}, where $M$ is now known to be a topologically finite hyperbolic 3--manifold or the interior of a compact Seifert fibred space with boundary. If the triangulation is not essential, it is also not strongly essential. 

Hence suppose the algorithm shows the triangulation to be essential. To test whether the triangulation is strongly essential, we need to perform the following additional test. Suppose $\gamma$ and $\delta$ are edge paths in $\widehat{M},$ with $\gamma$ running from vertex $i(\gamma) = v_1$ to vertex $t(\gamma)= v_2,$ and $\delta$ running from $i(\delta)=v_2$ to $t(\delta)=v_1.$ Again consider the intersections of $\gamma$ and $\delta$ with $M^c$ and denote the first point of intersection of $\gamma$ with $\partial M^c$ by $\gamma(a_\gamma)$ and the second by $\gamma(b_\gamma).$  Similarly for $\delta.$ 

Then choose a path $\rho_2$ on $\partial M^c$ from $\gamma(b_\gamma)$ to $\delta(a_\delta)$ and a path $\rho_1$ from $\delta(b_\delta)$ to $\gamma(a_\gamma).$ Then $\gamma' = (\gamma\mid_{[a_\gamma, b_\gamma]}) \star \rho_2 \star (\delta\mid_{[a_\delta, b_\delta]}) \star \rho_1$ is a loop in the 1--skeleton of $M^c$ based at $\gamma(a_\gamma).$ Denote the boundary component of $M^c$ containing $\gamma(a_\gamma)$ by $B_1$ and the component containing $\gamma(b_\gamma)$ by $B_2$ (they are possibly the same). Then let $H_1 = \im(\pi_1(B_1, \gamma(a_\gamma)) \to \pi_1(M^c, \gamma(a_\gamma)))$ and $H_2 = \im(\pi_1(B_2, \gamma(b_\gamma)) \to \pi_1(M^c, \gamma(a_\gamma))),$ where for the second inclusion map the path $\gamma\mid_{[a_\gamma, b_\gamma]}$ joining $\gamma(a_\gamma)$ and $\gamma(b_\gamma)$ is chosen.

Lemma~\ref{lem:parallel ideal edges in core} implies that $\gamma$ and $-\delta$ are related by an admissible path homotopy if and only if $[\gamma'] \in H_2H_1$. So we have reduced the problem of deciding whether there is an admissible path homotopy between $\gamma$ and $-\delta$ to testing whether a certain, algorithmically constructed loop represents an element in the double coset $H_2H_1.$ Since the groups $H_2$ and $H_1$ are finitely generated subgroups of $\pi_1(M),$ our hypotheses on $M$ allow us to solve this problem algorithmically (Proposition~\ref{pro:Double Coset Membership Problem}).

For any two edge paths $\gamma$, $\delta$ in the triangulation, we now need to apply the above algorithm to  $\gamma$ and $\delta$ if $i(\gamma) = t(\delta)$ and $i(\delta) = t(\gamma);$ and to $\gamma$ and $-\delta$ if $i(\gamma) = t(-\delta)$ and $i(-\delta) = t(\gamma).$ Then the triangulation is strongly essential if and only if the none of the constructed loops represents an element in the associated double coset.
\end{proof}

We will now bootstrap the above results to obtain a more general result.

\begin{Coro}
Given an ideal triangulation of the interior of a compact, irreducible, orientable, connected 3--manifold with incompressible boundary consisting of a disjoint union of tori, there is an algorithm to test whether it is strongly essential.
\end{Coro}

\begin{proof}
We first determine whether the triangulation is essential using Corollary~\ref{cor:algo ideal is essential}.  Suppose it is. The next step determines a system of tori giving the JSJ decomposition of the 3--manifold $M.$ This is done using work of Jaco and Tollefson~\cite{JT} as follows. The compact core $M^c$ is triangulated by construction. Algorithm 8.2 of \cite{JT} determines the collection of JSJ tori (noting that annuli do not arise because the ends of $M$ are all tori) as normal surfaces with respect to the finer triangulation. In particular, any edge in the ideal triangulation of $M$ meets these JSJ tori transversely. The JSJ tori divide the tetrahedra in $M^c$ into polytopes, and this finer cell decomposition of $M^c$ gives natural inclusion maps of the fundamental groups of the JSJ tori and the JSJ pieces into $\pi_1(M)$ (using appropriately chosen basepoints in the 0--skeleton and paths in the 1--skeleton).

Suppose, as above, that $\gamma$ and $\delta$ are edge paths with respect to the initial triangulation of  $\widehat{M},$ with $\gamma$ running from vertex $v_1$ to vertex $v_2,$ and $\delta$ running from $v_2$ to $v_1.$ Denote the intersections of $\gamma$ and $\delta$ with $M^c$ with the same symbols, and as above denote the boundary component linking $v_k$ by $B_k.$ Choose an orientation for each torus $T$ in the JSJ decomposition, and check whether the algebraic intersection number of $\gamma$ with $T$ is the negative of the intersection number of $\delta$ with $T.$ If the answer is no for any JSJ torus, then there is no admissible homotopy between $\gamma$ and $-\delta$. Otherwise, we will first analyse the intersections of the edges with the JSJ pieces $M_j.$ 

We start with some preliminary observations.
The conditions determining the properties essential and strongly essential with respect to the compact core apply to the intersection of $\gamma$ (resp.\thinspace $\delta$) with each piece in the JSJ decomposition of $M^c.$ 
We will count these components of $\gamma$ starting at the boundary component $B_1$ of $M_c,$ and write $\gamma = \gamma_1\star \gamma_2 \star \dots$ with $\gamma_j \subset M_j,$ where $M_j$ is a JSJ piece (these pieces are not necessarily distinct). Note that $\gamma_j$ may have endpoints on the same JSJ torus, but in this case, they are distinct as otherwise $\gamma_j$ would be a simple closed loop in $M^c.$

If the connected component $\gamma_j$ is not essential in $M_j,$ then it can be homotoped into $\partial M_j.$ Using the methods of \cite{AFW}, this can be done algorithmically using the solution of the word problem in $\pi_1(M_j)$, where we choose appropriate basepoints in the 0--skeleta and generating sets as well as paths for the inclusion maps in the 1--skeleta. Moreover, using the chosen generating sets for the JSJ tori, one can algorithmically replace $\gamma_j$ by the corresponding path $\tau$ in the 1--skeleton of the JSJ tori. Let $\gamma^\star$ denote the path obtained from $\gamma$ by replacing $\gamma_j$ with $\tau.$ Then $\gamma$ and $\gamma^\star$ are homotopic keeping their endpoints fixed, and we define the components of $\gamma^\star$ to be the components of $\gamma,$ except that $\gamma_{j-1}\star\tau\star\gamma_{j+1}$ is counted as one component since in this case $M_{j-1} =M _{j+1}$ and $\tau \subset \partial M_{j-1}.$ Hence we decrease the total number of components by two. Moreover, the intersection of $\gamma^\star$ with the JSJ tori is still transverse at the 
endpoints of each component of $\gamma^\star$ (even though $\gamma^\star$ is no longer transverse to them). We then repeat the above step with $\gamma^\star,$ and denote the resulting properly embedded arc in $M^c$ again by $\gamma^\star,$ until each component of $\gamma^\star$ is essential. Similarly for $\delta,$ resulting in $\delta^\star.$

If there is an admissible homotopy taking $\gamma$ to $-\delta$ in $\widehat{M},$ then there is an admissible homotopy between $\gamma^\star$ and $-\delta^\star$ in $M^c$ and they have the same number of components. Moreover, the  $k$-	-th component of $\gamma^\star$ has endpoints on the same tori and with the same respective signs as the $k$th component of $-\delta^\star.$ Hence we check this property of $\gamma^\star$ and $\delta^\star$ first, and if it is not satisfied, then there is no admissible homotopy taking $\gamma$ to $-\delta.$

Hence assume that the intersections match up. For each $k,$ we check whether the  $k$th component of $\gamma^\star$ is parallel to the  $k$th component of $\delta^\star$ in the JSJ piece that contains them. If the answer is no for any $k,$ then there is no admissible homotopy taking $\gamma$ to $-\delta.$ Hence assume that the components are pairwise parallel. 

Then there is a homotopy taking $\gamma^\star$ to the union of $\delta^\star$ with possibly some loops on the JSJ tori inserted. This is because in checking parallelism in the JSJ pieces, we are free to move the ends of arcs on the JSJ tori, which is no longer possible when considering the edges in $M^c.$ However, choosing basepoints appropriately, the Double Coset Membership Test in each JSJ piece will return the corresponding elements in the JSJ tori. Since for each $k$ the $k$th components are parallel, we can choose the same paths (but with opposite orientation) on the JSJ tori to connect the terminal points of a pair of parallel components as well as the initial points of the next pair of parallel components, and hence write the loop $\gamma^\star \star \rho_1 \star \delta^\star \star \rho_2$ based at $a_\gamma$ as a product of elements in the fundamental groups of the JSJ tori and the boundary components $B_1$ and $B_2$ (by choosing appropriate base points and inclusion homomorphisms). It then suffices to check that the resulting element is contained in the double coset $H_2H_1$ as above. There is an admissible homotopy between $\gamma$ and $-\delta$ if and only if the answer to this test is positive. 

Finally, if $v_1=v_2$ and there is no admissible path homotopy found between $\gamma$ and $-\delta,$ we also need to run the whole algorithm with $\delta$ and $-\delta$ interchanged.
\end{proof}

 
\providecommand{\bysame}{\leavevmode\hbox to3em{\hrulefill}\thinspace}
\providecommand{\href}[2]{#2}

\end{document}